\documentclass{amsart}
\usepackage{marginnote}
\usepackage{bm}
\usepackage{tikz-cd}
\usepackage{tikz}\usepackage{float}
\usepackage{fancyhdr}
\usepackage[T1]{fontenc}
\usepackage{mathrsfs}
\usepackage[hidelinks]{hyperref}
\usepackage{setspace}
\usepackage{amstext,amsfonts} 
\usepackage{array}
\usepackage{comment}
\usetikzlibrary{calc}
\allowdisplaybreaks
\theoremstyle{plain}

\newtheorem{thm}{Theorem}[section]
\newtheorem{lem}[thm]{Lemma}
\newtheorem{prop}[thm]{Proposition}
\newtheorem{si}[thm]{Situation}
\newtheorem{conj}[thm]{Conjecture}

\newtheorem{q}{Question}

\newtheorem{cor}[thm]{Corollary}
\theoremstyle{definition}
\newtheorem{defn}[thm]{Definition}
\theoremstyle{remark}
\newtheorem{rem}[thm]{Remark}
\theoremstyle{definition}
\newtheorem{ex}[thm]{Example}

\newtheorem{nt}[thm]{Notation}

\usepackage{amsmath,amssymb,amsfonts,amstext,amsthm}
\usepackage{url}
\usepackage{graphicx}
\input xy
\xyoption{all}
\usepackage[all]{xy}

\graphicspath{{images/}{../images/}}
\usepackage{subfiles}
\usepackage{tikz}

\newcommand{\nwc}{\newcommand}

\nwc{\ba}{\mathbb{A}}
\nwc{\aaa}{\mathcal{F}}
\nwc{\aap}{\mathcal{F}_{P}}
\nwc{\al}{\alpha}
\nwc{\bs}{\backslash}
\nwc{\C}{\mathbb{C}}
\nwc{\cb}{\overline{C}}
\nwc{\ccc}{\mathfrak{c}}
\nwc{\cin}{\textbf{(v)}}
\nwc{\cl}{C'}
\nwc{\cp}{\mathcal{C}_{P}}
\nwc{\cpll}{\mathfrak{c}_{P'}}
\nwc{\ct}{\widetilde{C}}

\nwc{\dd}{\mathcal{L}}
\nwc{\ddd}{\mathfrak{d}}
\nwc{\ddl}{\mathcal{L}'}
\nwc{\dlp}{\delta_{P}}
\nwc{\doi}{\textbf{(ii)}}
\nwc{\g}{\mathfrak g}
\nwc{\enq}{$$}

\nwc{\fff}{\mathcal{F}}
\nwc{\ffp}{\mathcal{F}_{P}}
\nwc{\ffq}{\mathcal{F}_{Q}}
\nwc{\ffl}{\mathcal{F}'}

\nwc{\G}{\mathcal{G}}
\nwc{\Ga}{\Gamma}
\nwc{\gtl}{\widetilde{g}}
\nwc{\mO}{\mathcal{O}}
\nwc{\mF}{\mathcal{F}}
\nwc{\hra}{\hookrightarrow}
\nwc{\hua}{h^{1}(C,\aaa )}
\nwc{\cM}{\mathscr{M}}
\nwc{\kk}{{\rm K}}
\nwc{\sg}{\sigma}
\nwc{\llb}{\mathcal{L}}

\nwc{\mb}{\mathbb}
\nwc{\mc}{\mathcal}
\nwc{\mm}{\mathfrak{m}}
\nwc{\mmp}{\mathfrak{m}_{P}}
\nwc{\mpd}{\mathfrak{m}_{P}^{2}}
\nwc{\ov}{\overline}
\nwc{\nn}{\mathbb{N}}
\nwc{\ob}{\overline{\mathcal{O}}}
\nwc{\obr}{\mathcal{O}^*}
\nwc{\obp}{\overline{\mathcal{O}}_P}
\nwc{\och}{\mathcal{O}_{\hat{C}}}
\nwc{\oh}{\hat{\mathcal{O}}}
\nwc{\ohp}{\hat{\mathcal{O}}_{P}}
\nwc{\ol}{\mathcal{O}'}
\nwc{\oma}{\Omega (\mathfrak{a})}
\nwc{\omo}{\Omega (\mathcal{O})}
\nwc{\oo}{\mathcal{O}}
\nwc{\op}{\oplus}
\nwc{\opc}{\mathcal{O}_{P,C}}
\nwc{\oph}{\hat{\mathcal{O}}_{P}}
\nwc{\opl}{\mathcal{O}_{P}'}
\nwc{\oplc}{\mathcal{O}_{P,C}'}
\nwc{\ti}{\widetilde}
\nwc{\opll}{\mathcal{O}_{P'}}
\nwc{\opt}{\tilde{\mathcal{O}}_{P}}
\nwc{\optt}{{\mathcal{O}}_{\tilde{P}}}
\nwc{\oq}{\mathcal{O}_{Q}}
\nwc{\oqt}{\tilde{\mathcal{O}}_{Q}}
\nwc{\ot}{\otimes}
\nwc{\overop}{\bar{\oo}_{P}}
\nwc{\qq}{\mathbb{Q}}
\nwc{\PP}{\mathbb{P}}
\nwc{\pb}{\overline{P}}
\nwc{\pbb}{P^*}
\nwc{\pbi}{\overline{P_{i}}}
\nwc{\pbr}{\overline{P_{r}}}
\nwc{\cd}{\cdot}
\nwc{\pgmd}{\mathbb{P}^{g+2}}
\nwc{\pgmu}{\mathbb{P}^{g+1}}
\nwc{\ph}{\hat{P}}
\nwc{\pp}{\mathbb{P}}
\nwc{\NN}{\mathbb{N}}
\nwc{\QQ}{\mathbb{Q}}
\nwc{\prv}{\noindent\textbook{Proof}:}
\nwc{\pt}{\widetilde{P}}
\nwc{\ptl}{\tilde{P}}
\nwc{\pum}{\mathbb{P}^{1}}
\nwc{\fa}{\forall}
\nwc{\qh}{\hat{Q}}
\nwc{\qtl}{\tilde{Q}}
\nwc{\qua}{\textbf{(iv)}}
\nwc{\p}{\partial}
\nwc{\ra}{\rightarrow}
\nwc{\rh}{\hat{R}}

\nwc{\sei}{\textbf{(vi)}}
\nwc{\sep}{\beq\ast\ \ast\ \ast\enq}
\nwc{\sig}{\sigma}
\nwc{\Sig}{\Sigma}
\nwc{\ssp}{S_{P}}
\nwc{\sss}{{\rm S}}
\nwc{\cA}{\mathcal{A}}
\nwc{\cB}{\mathcal{B}}
\nwc{\cC}{\mathcal{C}}
\nwc{\cD}{\mathcal{D}}
\nwc{\cE}{\mathcal{E}}
\nwc{\cF}{\mathcal{F}}
\nwc{\cG}{\mathcal{G}}
\nwc{\cH}{\mathcal{H}}
\nwc{\cI}{\mathcal{I}}
\nwc{\cJ}{\mathcal{J}}
\nwc{\cK}{\mathcal{K}}
\nwc{\cL}{\mathcal{L}}
\nwc{\CM}{\mathcal{M}}
\nwc{\cN}{\mathcal{N}}
\nwc{\sM}{\mathscr{M}}
\nwc{\fC}{\mathfrak{C}}
\nwc{\tre}{\textbf{(iii)}}
\nwc{\fJ}{\mathfrak J}
\nwc{\um}{\textbf{(i)}}

\nwc{\vpb}{v_{\overline{P}}}
\nwc{\vtxp}{\widetilde{V}_{x,P}}
\nwc{\vxp}{V_{x,P}}

\let \wt=\widetilde

\nwc{\wh}{\hat{\omega}}
\nwc{\whp}{\hat{\omega}_{P}}
\nwc{\woch}{\omega\cdot\mathcal{O}_{\hat{C}}}
\nwc{\woh}{\omega\cdot\hat{\mathcal{O}}}
\nwc{\ww}{\omega}
\nwc{\wwb}{\omega^*}
\nwc{\wwct}{\omega _{\widetilde{C}}}
\nwc{\wwh}{\widehat{\omega}}
\nwc{\wwhp}{\widehat{\omega}_P}
\nwc{\wwp}{\omega _{P}}
\nwc{\wwt}{\widetilde{\omega}}
\nwc{\wwtp}{\widetilde{\omega}_P}
\nwc{\lag}{\langle}
\nwc{\rag}{\rangle}
\nwc{\ZZ}{\mathbb{Z}}
\nwc{\hh}{,\hdots,}

\let \ep=\epsilon
\let \ga=\gamma
\let \sub=\subset

\let \al=\alpha
\let \pr=\prime
\let \la=\lambda

\let \fr=\frac

\DeclareMathOperator\id{Id}

\DeclareMathOperator\supp{Supp}

\DeclareMathOperator\sy{Sym}

\DeclareMathOperator\spec{spec}

\DeclareMathOperator\Gal{Gal}
\DeclareMathOperator\dv{div}

\DeclareMathOperator\h{Hom}
\DeclareMathOperator\Gh{G-Hom}

\DeclareMathOperator\codim{codim}

\DeclareMathOperator\Jac{Jac}
\nwc {\Sp}{\underline{\spec}}

\calclayout

\DeclareRobustCommand{\SkipTocEntry}[5]{}

\begin{document}

\title{Weierstrass semigroups from cyclic covers of hyperelliptic curves}
\author{Ethan Cotterill, Nathan Pflueger, and Naizhen Zhang}
\begin{abstract}
The {\it Weierstrass semigroup} of pole orders of meromorphic functions in a point $p$ of a smooth algebraic curve $C$ is a classical object of study; a celebrated problem of Hurwitz is to characterize which semigroups ${\rm S} \sub \mb{N}$ with finite complement are {\it realizable} as Weierstrass semigroups ${\rm S}= {\rm S}(C,p)$. In this note, we establish realizability results for cyclic covers $\pi: (C,p) \ra (B,q)$ of hyperelliptic targets $B$ marked in hyperelliptic Weierstrass points; and we show that realizability is dictated by the behavior under $j$-fold multiplication of certain divisor classes in hyperelliptic Jacobians naturally associated to our cyclic covers, as $j$ ranges over all natural numbers.
\end{abstract}

\thanks{During much of the preparation of this work, Naizhen Zhang was supported by the DFG Priority Programme 2026 ``Geometry at infinity''.}
\address[Ethan Cotterill]{Imecc, Unicamp, Rua S\'ergio Buarque de Holanda, 651, 13.083-859, Campinas, SP, Brazil}
\email{cotterill.ethan@gmail.com}
\address[Nathan Pflueger]{Department of Mathematics and Statistics, Amherst College, Amherst, MA 01002}
\email{npflueger@amherst.edu}
\address[Naizhen Zhang]{Fairleigh Dickinson University Vancouver Campus, 842 Cambie St, Vancouver, BC V6B 2P6}
\email{n.zhang1@fdu.edu}
\maketitle
\tableofcontents
\section{Introduction}
{\it Hurwitz spaces} of covers $C \ra B$ of algebraic curves of fixed genus 
have long played an important role in algebraic geometry and number theory. {\it Cyclic} covers are distinguished by their simplicity; indeed, the geometry of a cyclic cover is controlled by its branch locus on the target curve $B$. It is natural to ask, in particular, about the behavior of {\it Weierstrass semigroups} in ramification points of a cyclic cover; this is a cyclic cover version of Hurwitz' celebrated {\it realization problem} for Weierstrass semigroups. Resolving it, however, is untenable without introducing further hypotheses. In this paper, we address the realization problem for ramification points $p$ of {\it marked} cyclic covers $(C,p) \ra (B,q)$ whose targets are hyperelliptic curves marked in hyperelliptic Weierstrass points. 

\medskip
Broadly speaking, this paper is in two parts. In the first, spanned by sections~\ref{sec:background} and \ref{sec:Weierstrass-realizability}, we develop the basic combinatorial machinery required to specify cyclic covers of hyperelliptic curves, and Weierstrass semigroups in the preimages of hyperelliptic Weierstrass points. We then apply this machinery to establish explicit restrictions on numerical semigroups that arise as Weierstrass semigroups associated with cyclic covers; we nearly resolve the Weierstrass-realizability problem for cyclic {\it triple} covers; and we establish that the Weierstrass-realizability problem for cyclic covers is equivalent to a realizability problem for {\it multiplication profiles} in hyperelliptic Jacobians, which describe the multiplicative behavior of divisor classes in the Jacobian of $B$ that are naturally extracted from the branch divisors of our covers. The second half of this paper, spanned by sections~\ref{sec:mult_profile} and \ref{sec:explicit_realizability}, is a close study of
the multiplication profiles of elements $[q^{\pr}-q]$ of the Jacobian that belong to image of $(B,q)$ under the Abel--Jacobi map. In section~\ref{sec:mult_profile} we establish a number of arithmetic constraints on these. We then focus in more detail on cyclic covers whose underlying branch divisors represent Abel--Jacobi images of torsion classes. The study of torsion points in hyperelliptic Jacobians that arise as Abel--Jacobi images is an active area of inquiry in its own right, with links to Diophantine solutions of generalized Pell equations (see \cite{SerreBourbaki,Zannier}) and the dynamics of billiards in ellipsoids (see \cite{GHponcelet,Dragovic}).

\medskip
It is worth noting that the question of which numerical semigroups are realizable as Weierstrass semigroups remains wide open in general, despite many special cases in the literature. A broad class of semigroups where realizability is known are those of low \emph{effective weight}, as defined in \cite{Pfl}: all genus-$g$ semigroups with $\operatorname{ewt}({\rm S}) \leq g-1$ are known to be realizable. The geometric relevance of $\operatorname{ewt}({\rm S})$, as established in {\it loc. cit.}, is that $3g-2-\operatorname{ewt}({\rm S})$ is a lower bound for the dimension of the moduli space $\mc{M}_{g,1}^{\rm S}$ of (smooth) curves of genus $g$ marked in points with Weierstrass semigroup ${\rm S}$. Nearly all of the multiplication profiles constructed in this paper result in realizable Weierstrass semigroups of effective weight larger than $g$; so these semigroups occupy as-yet-uncharted territory.

\medskip
Numerical semigroups ${\rm S}$ that may arise from our $N$-fold cyclic covers are all of multiplicity $2N$. As such, ${\rm S}$ has a {\it standard basis} $(e_0,\dots,e_{2N-1})$ comprised of nonzero elements of ${\rm S}$ that are minimal in their respective equivalence classes modulo $2N$.
Our main results on Weierstrass-realizability may be summarized as follows.

\begin{thm}\label{thm:main1}(Theorem \ref{thm:realizability3}.)
 An $(3,\gamma)$-hyperelliptic semigroup ${\rm S}$ of genus $g\ge 5\ga$ with standard basis $\{6,e_1\hh e_5\}$ is Weierstrass-realizable as a triple cover of a genus-$\ga$ hyperelliptic curve provided either $e_4<e_1$ or $e_4>e_1\ge 4g-6\ga+7-2\min\{e_2,e_5\}$.
\end{thm}
\begin{thm}\label{thm:main2}(Theorem \ref{thm:torsion}.)
    Let ${\rm S}$ be an $(N,\ga)$-hyperelliptic semigroup  of genus $g\ge(2N-1)\ga$ with standard basis $\{2N,e_1\hh e_{2N-1}\}$ modulo $2N$, and suppose $d=\frac{(2g-2)-N(2\ga-2)}{N(N-1)}$ is an integer. 
    The semigroup ${\rm S}$ is Weierstrass-realizable via a degree-$N$ cover of a genus-$\ga$ hyperelliptic curve if its $\ep$-vector (Definition \ref{defn:ep_vec}) is equal to one of the following:
 \begin{enumerate}
     \item any potential multiplication profile of a torsion point whose order is either among $2,2\ga+1,2\ga+2,2\ga+4,2\ga+6$ or is $2\ga+7$, with $\ga\ge 4$;
     \item three out of five potential multiplication profiles of a torsion point of order $2\ga+8$ (listed in Example~\ref{ex:2ga+8}); or
     \item one out of three potential multiplication profiles of a torsion point of order $2\ga+9$, and $\ga\ge 4$ (listed in Example~\ref{ex:2ga+9}).
 \end{enumerate}   
\end{thm}

\begin{thm}\label{thm:eff_wt}(Proposition \ref{prop:non-secundiveness})
    The $(3,\ga)$-hyperelliptic semigroups realized in Theorem \ref{thm:main1} have effective weights greater than or equal to their respective genera $g$ whenever $g\ge 6\ga+7$. More generally, any feasible $(N,\ga)$-hyperelliptic semigroup of multiplicity $2N$ and genus $g$ (Definition \ref{defn:feasible}) has effective weight at least $g$ whenever $g>\frac{N}{2}(4\ga+3)+1$.
\end{thm}
Our results have the following moduli-theoretic consequences:
\begin{cor}
    For any numerical semigroup ${\rm S}$ listed in Theorems~\ref{thm:main1}, \ref{thm:main2} and \ref{thm:eff_wt}, the moduli space $\mc{M}_{g,1}^{\rm S}$ is non-empty and contains at least one component whose dimension is at least $3g-2-\operatorname{ewt}({\rm S})$. Furthermore, many of our examples in Theorem \ref{thm:main2} satisfy $\operatorname{ewt}({\rm S})> 3g-2 $, thereby furnishing a large class of examples where ${\rm S}$ is realizable and $\codim\mc{M}_{g,1}^{\rm S}<\operatorname{ewt}({\rm S})$. See \url{https://github.com/npflueger/cyclicCoverSemigroups/blob/main/listStaircase-output.txt}.
\end{cor}
\addtocontents{toc}{\SkipTocEntry}
\subsection*{Roadmap}

A detailed synopsis of the material following this introduction is as follows. The main result of section~\ref{sec:background} is {\bf Proposition~\ref{prop:push}}, which explicitly characterizes the decomposition of the pushforward of the (twisted) structure sheaf $\mc{O}_C(mp)$ under the Galois quotient of marked curves $(C,p) \ra (B,q)$ by a finite cyclic group; as such, it determines how the Weierstrass semigroup in $p$ is controlled cohomologically on the hyperelliptic target $(B,q)$. Section~\ref{sec:background} also includes a review of Mumford's theory of {\it reduced} presentations of divisors on a hyperelliptic curve, and Cantor's algorithm for computing these. In this paper, we always reduce with respect to a fixed Weierstrass point $q$, and use presentations that split as an effective divisor minus a multiple of $q$. 

\medskip
In section~\ref{sec:Weierstrass-realizability}, we specialize to cyclic covers in which the target is a hyperelliptic curve $B$ marked in a hyperelliptic Weierstrass point $q$.
In {\bf Theorem~\ref{gap_vector_thm}}, and specifically in Equation \eqref{cyclic_cover_standard_basis_restriction}, we establish, via a cohomological argument, that the value of every pairwise sum $e_i+e_{N+i}$, $i=1,\dots,N-1$ of components in the standard basis $(2N,e_1,\dots,e_{2N-1})$ of the Weierstrass semigroup in $p$ is fixed; so the standard basis is determined by $\min\{e_1,e_{N+1}\}\hh \min\{e_{N-1},e_{2N-1}\}$. Elements of the standard basis satisfy inequalities that arise from the additive structure of ${\rm S}$; taken together, these prescribe a {\it feasible set} $F(N)$ of potentially-realizable values for $a_j=\frac{\min\{e_{N-j},e_{2N-j}\}+j}{N}$, $j=1,\dots,N-1$. In {\bf Proposition~\ref{prop:reduced_pair}}, we reformulate Weierstrass-realizability as the existence of an $(N-1)$-tuple of effective divisors $(E_1,\dots,E_{N-1})$ in which each $E_j$ is the effective part of the $q$-reduced representative of $jE_1$ for every $j$; the degrees of the $E_j$ comprise the \textit{multiplication profile} of $E_1$.
{\bf Theorem~\ref{thm:realizability}} gives a concrete application of Proposition~\ref{prop:reduced_pair}: we show that when $N=3$, a proper nonempty subset of $F(3)$ indexes semigroups that are Weierstrass-realizable via cyclic covers.

\medskip
In section~\ref{sec:mult_profile} we begin a close study of the multiplication profiles of point classes $[q^{\pr}-q]$.
{\bf Proposition~\ref{prop:profile_conditions}} gives an explicit list of constraints for a sequence of integers to arise from some point class, which we prove to be necessary and conjecture to be sufficient; while {\bf Proposition~\ref{prop:coeff_prof}} establishes that every multiplication profile is uniquely characterized by the vanishing or non-vanishing of explicit determinants in the coefficients of the expansion of $y=\sqrt{f(x)}$ near $x=0$ as a power series in $x$. We use these conditions, in turn to construct an explicit quasi-projective parameter space for marked hyperelliptic curves equipped with $N$-torsion points.

\medskip
In {\bf Examples~\ref{ex:(2ga+1)-torsion}-\ref{ex:2ga+9}} of section~\ref{sec:explicit_realizability} we leverage Mumford presentations, Cantor's reduction algorithm, and earlier work of Flynn \cite{flynn1991sequences} to produce explicit models of hyperelliptic curves with torsion orders $N \leq 2\ga+9$; we managed to realize 14 out of 20 possible multiplication profiles. In {\bf Example~\ref{ex:fibonacci_numerology}}, we combinatorially analyze those sequences of numbers eligible to be multiplication profiles of $N$-torsion point classes, and show that in each respective parity, the number of eligible sequences is a Fibonacci number.
In {\bf Propositions~\ref{prop:CT_presentation_and_ewt} and \ref{prop:staircase_minimal_generation}} we combinatorially analyze two distinguished multiplication profiles: the all-zeroes profile, corresponding to a semigroup ${\rm U}_0$ studied earlier by Carvalho and Torres, and a profile of {\it staircase} type, which indexes a semigroup ${\rm S}_N$ for every $N \geq 2$. We compute the effective weight of ${\rm U}_0$ and a minimal presentation for ${\rm S}_N$, respectively, under appropriate hypotheses. And in {\bf Proposition~\ref{prop:non-secundiveness}} we show that modulo suitable numerical hypotheses, any $(N,\ga)$-hyperelliptic semigroup ${\rm S}$ of genus $g$ 
that arises via our cyclic cover construction is {\it non-secundive} in the sense of \cite[Def. 5.1]{Pfl}, and therefore has effective weight at least $g$. We have also explicitly computed the effective weights of an extensive list of numerical semigroups via computer; our code may be consulted at \url{https://github.com/npflueger/cyclicCoverSemigroups} and our
examples at \url{https://github.com/npflueger/cyclicCoverSemigroups/blob/main/listSemigroups-output.txt} and \url{https://github.com/npflueger/cyclicCoverSemigroups/blob/main/listStaircase-output.txt} respectively. 

\medskip
Finally, Appendix~\ref{sec:separability} is a collection of separability results for polynomials that are needed in order to certify that our models are smooth, while Appendix~\ref{sec:flynn} is a detailed treatment of (our adaptations of) Flynn's models.

\subsection{Conventions}\addtocontents{toc}{\protect\setcounter{tocdepth}{1}} In this paper, we work over an algebraically closed field $k$ of characteristic zero. A variety over $k$ will always mean an integral, separated scheme of finite type over $k$ unless otherwise specified.

\subsection{Acknowledgements} We thank Joe Harris 
for illuminating conversations, and the anonymous referee for comments that have helped us improve the exposition. Fernando Torres invited the first author to speak in Campinas on this topic in early 2020, and sadly left us before this project's completion. We dedicate this paper to his memory.

\section{Background}\label{sec:background}
\subsection{Cyclic covers of algebraic curves}\addtocontents{toc}{\protect\setcounter{tocdepth}{1}}
In this section, we review the basic theory of cyclic covers of smooth projective algebraic curves. 
\begin{defn}[Cf. \cite{pardini1991abelian}, \cite{pardini1998period}]\label{abelian_cover}
Given a smooth variety $Y$ and a finite abelian group $G$, a $G$-\textit{cover} of $Y$ is a finite, surjective morphism $\pi: X\to Y$ for which $X$ is normal, $K(X)/K(Y)$ is Galois, and $\Gal(K(X)/K(Y))=G$.
\end{defn}
Definition~\ref{abelian_cover} has a number of immediate consequences. First of all, finiteness of $\pi$ implies that it is affine, and therefore $X\cong\Sp(\pi_*\mO_X)$. Secondly, $G$ acts on $X$ and finiteness of $G$ implies $Y=X/G$; moreover, $G$ acts transitively on the fibers of $\pi$. Thirdly, whenever $\pi$ is flat (e.g., when $\pi$ is a finite map between smooth projective curves), $\pi_*\mO_X$ is a locally-free sheaf on $Y$ carrying the structure of an $\mO_Y$-algebra with a $G$-action. It is therefore useful to work in the category of $G$-sheaves; see \cite[Sec. 4]{bridgeland2001mckay}.

In our study of degree-$N$ cyclic covers $\pi:C\to B$ of curves, $G$ will always be the cyclic Galois group $\Gal(K(C)/K(B))\cong\mu_N=\langle\xi\mid\xi^N=1\rangle$, which acts trivially on $B$ and permutes points in the fibers of $\pi$. In particular, $\pi$ is $G$-equivariant and for any $G$-line bundle $L$ on $C$ (resp. any line bundle $M$ on $B$), $L$ and $\pi_*L$ (resp. $M$ and $\pi^*M$) are $G$-sheaves on $C$ (resp. $B$) and $B$ (resp. $C$). In this context, we have
\begin{equation}\label{eqn:nat_iso}
\Gh_{\mO_C}(\pi^*M,L)\cong \Gh_{\mO_B}(M,\pi_*L).
\end{equation}


\begin{ex}\label{ex:tot_ramf}
Let $\pi:C\to B$ be a degree-$N$ cyclic cover of smooth projective curves equipped with a totally ramified point $p\in C$. Then $\mO_C(mp)$ is $\mu_N$-invariant, as $p$ is fixed by the action of the Galois group. As a result, $\pi_*\mO_C(mp)$ is a locally-free $G$-sheaf on $B$ and hence decomposes as a direct sum of line bundles
\[\pi_*\mO_C(mp)=\bigoplus_{i=0}^{N-1}L_i\]
where $L_i=\ker(\la^{\pi_*\mO_C(mp)}_{\xi}-\eta^{i}\cd\id)$, for some fixed 
primitive $N$-th root of unity $\eta$.
Given our assumptions on the ground field as well as on the extension $K(C)/K(B)$, we have $K(C)=K(B)[x]/(x^N-\theta)$ for some $\theta\in K(B)$. In particular, $K(C)$ decomposes as $K(C)=\bigoplus_{i=0}^{N-1}(K(B)\cd x^i)$, where $K(B)\cd x^i$ is the eigenspace with eigenvalue $\eta^i$ for the action of our pre-chosen generator $\xi$ of the Galois group. It follows that $L_i(V)=\mO_C(mp)(\pi^{-1}(V))\cap (K(B)\cd x^i)$ for every open set $V\subset B$.
\end{ex}

\subsection{Covering data}
Hereafter we will focus on simple cyclic covers of smooth projective curves, each of which is determined by a pair $(L,D)$ for which $L^{\ot N}\cong\mO_B(D)$. More concretely, there is an $\mO_Y$-algebra structure on $\bigoplus_{i=0}^{N-1}L^{-i}$ given by the natural morphisms \[L^{k-N}\cong L^{k}\ot\mO_B(-D)\to L^{k}\ot\mO_B\cong L^k.\] 

\begin{lem}\label{lem:normalization}
Let $D$ be an effective divisor, $L$ be an invertible sheaf on a smooth projective curve $B$, such that $L^{\ot N}\cong \mO_B(D)$. 
There is a degree-$N$ cyclic cover $\pi:C\to B$ of smooth projective curves for which
\[\pi_*\mO_C\cong \bigoplus_{i=0}^{N-1}L^{-i}\bigg(\bigg[\frac{iD}{N}\bigg]\bigg)\]
which is the normalization of the cover
\[\Sp(\bigoplus_{i=0}^{N-1}L^{-i})\to B.\]
\end{lem}
Here $[\cdot]$ denotes the integral part of a $\qq$-Cartier divisor. The cover \eqref{lem:normalization} is sometimes called the $N$\textit{-th extraction} of $D$ (see \cite{esnault1982revetements}).
\begin{proof}
This is a special case of \cite[Lem. 1.8]{esnault1982revetements}.  
\end{proof}

\begin{rem}
We shall refer to pairs $(L,D)$ such that $L^{\ot N}\cong\mO_B(D)$ on curves as \textit{covering data}.
\end{rem}

\begin{thm}[Thm. 3.3, \cite{morrison1986galois}]\label{thm:foundation}
Associated with each collection of covering data $(L,D=\sum a_kq_k)$, $1\le a_k<N$ for which $L^{\ot N}\cong\mO_B(D)$, there is a degree-$N$ cyclic cover $\pi:C\to B$ branched over $D_{\text{red}}$, where $\pi^{-1}(q_k)$ consists of $\gcd(a_k,N)$ many points. Conversely, any degree-$N$ cyclic cover $\pi:C\to B$ may be constructed from some covering data $(L,D=\sum a_kq_k)$ using the method of Lemma~\ref{lem:normalization}; the same cover may also be constructed using the covering data $(L^{(i)},D^{(i)})$, where $i<N$, $\gcd(i,N)=1$ and 
\[L^{(i)}=L^i\bigg(-\bigg[\frac{iD}{N}\bigg]\bigg),D^{(i)}=\sum a^{(i)}_kq_k\]
with $a^{(i)}_k\equiv ia_k$ (mod $N$) and $1\le a^{(i)}<N$. 
\end{thm}

\begin{proof}
The construction of a cyclic cover from covering data $(L,D)$ is the content of Lemma \ref{lem:normalization}. It is easy to see that $\pi$ is branched over $D_{\text{red}}$, as $\pi_*\mO_C|_{B\backslash D_{\text{red}}}\cong \mO_{B\backslash D_{\text{red}}}^{\op N}$. The fact that $\pi^{-1}(q_k)$ consists of $\gcd(a_k,N)$ many points follows from the fact that $K(C)/K(B)$ is a Kummer extension of function fields. \footnote{Recall also that in this case all ramification points over a given point $q_k$ in the branch locus have the same ramification index $N_k=\frac{N}{\gcd(a_k,N)}$.}

\medskip
Conversely, given any degree-$N$ cyclic cover $\pi:C\to B$, we have $\pi_*\mO_C\cong \mO_Y\oplus \bigoplus_{\chi} L_{\chi}$. Choose $L=L_{\chi}^{-1}$ for some $\chi$ generating $\widehat{\mu_N}$, together with a section $s$ of $L^{\ot N}$, and let $D=\dv(s)$; then up to isomorphism of covers, $(L,D)$ is a collection of covering data for $\pi$. Moreover, whenever $\text{gcd}(i,N)=1$ and $j \in [1,N]$ is such that $i\cd j\equiv 1$(mod $N$), we have $(L^{(i)})^{(j)}=L$ and $(D^{(i)})^{(j)}=D$. Thus, $(L^{(i)},D^{(i)})$ determines a collection of covering data for $\pi$ whenever $(L,D)$ does.
\end{proof}

\begin{rem}
In particular, $\pi$ is totally ramified at all ramification points whenever $N$ is prime.
\end{rem}

Now suppose $D=\sum a_kq_k$ and that $p$ is a ramification point of $\pi$ with image $\pi(p)=q_k$. Set $N_k:=\frac{N}{\gcd(a_k,N)}$, $a^{\pr}_k:=\frac{a_k}{\gcd(a_k,N)}$. Recall that $K(C)=K(B)[x]/(x^N-\theta)$; so $x^n= \theta$. As $v_{q_k}(\theta)=a_k$, it follows that $v_p(\theta)=a_kN_k$ and therefore $v_p(x)=a^{\pr}_k$.

\medskip
On the other hand, recall that $L^{-i}([\frac{iD}{N}])$ is the eigensheaf associated to $\eta^i$, where $\eta$ is our pre-chosen primitive $N$-th root; so $\pi^{*}L^{-i}([\frac{iD}{N}])$ is the invertible sub-sheaf of $\mO_C$ locally generated by some multiple of $x^i$. In particular, we deduce the following:

\begin{lem}\label{lem:summand_stalk}
Suppose that $\pi(p)=q_1$. The image of 
$\big(\pi^{*}L^{-i}([\frac{iD}{N}])\big)_p$ in $\mO_{C,p}$ is $\mathfrak{m}^b_p$, where $b$ is the unique integer in $[0,N_1-1]$ congruent to $ia'_1$ modulo $N_1$. 
\end{lem}
\begin{proof}
Elements in $\big(\pi^{*}L^{-i}([\frac{iD}{N}])\big)_p$ are represented by functions of the form $fx^i$, where $f\in K(B)$ is regular near $p$. Let $t_{q_1}$ be a uniformizer at $q_1\in B$. The desired conclusion follows from the facts that $v_p(\pi^{*}t_{q_1})=N_1$ and $v_p(x^i)=i\cdot a^{\pr}_1$.
\end{proof}

\begin{prop}\label{prop:push}
Suppose $\pi:C\to B$ is a cyclic cover constructed out of the covering data $(L,D=\sum a_iq_i)$ for which $L^{\ot N}\cong\mO_B(D)$. Fix $p\in C$ such that $q_1=\pi(p)$, and suppose $\pi$ is totally ramified at $p$. Then for every non-negative integer $m$, we have 
\begin{equation}
    \pi_*\mO_C(mp)=\bigoplus_{i=0}^{N-1}(L^{(i)})^{-1}\bigg(\bigg\lfloor\frac{m}{N}+\bigg\lag \frac{ia_1}{N}\bigg\rag \bigg\rfloor q_1\bigg)
\end{equation}
where $L^{(i)}=L^i(-[\frac{iD}{N}])$, for $i=0\hh N-1$, and $\lag c\rag$ denotes the fractional part of a positive rational number $c$. 
\end{prop}
\begin{proof}
As discussed in Example \ref{ex:tot_ramf}, $\pi_*\mO_C(mp)$ decomposes as a direct sum of line bundles $\pi_*\mO_C(mp)=\bigoplus_{i=0}^{N-1}L_i$, where $L_i$ is the eigensheaf on which a fixed generator $\xi$ of $\mu_N$ acts by multiplication-by-$\eta^i$, for some preselected primitive $N$-th root of unity $\eta$.    

\medskip
Now let $U=B\backslash\{q_1\}$. As $\mO_C(mp)|_{\pi^{-1}(U)}=\mO_C|_{\pi^{-1}(U)}$, we have $\pi_*\mO_C(mp)|_{U}=\pi_*\mO_C|_{U}=\bigoplus_{i=0}^{N-1}(L^{(i)})^{-1}|_U$. This means, in particular, that 
\[\pi_*\mO_C(mp)=\bigoplus_{i=0}^{N-1}(L^{(i)})^{-1}(n_iq_1)\]
for certain integers $n_i$, and it remains to determine these.
For this purpose, we use the natural isomorphism of Equation (\ref{eqn:nat_iso}), which yields
\[\mu_N-\h_{\mO_B}((L^{(i)})^{-1}(nq_1), \pi_*\mO_C(mp))\cong \mu_N-\h_{\mO_C}((\pi^*(L^{(i)})^{-1})(nNp), \mO_C(mp)).\] 
The upshot is that $n_i$ is the maximal integer $n$ for which there is a natural inclusion $\pi^*(L^{(i)})^{-1}(nq_1)\hookrightarrow\mO_C(mp)$ of $\mu_N$-sheaves. 

\medskip
According to Lemma \ref{lem:summand_stalk}, the image of $(\pi^*L^{(i)^{-1}})_p$ is $\mathfrak{m}^b_p$, where $b$ is the unique integer between $1$ and $N-1$ for which $b\equiv i\cdot a_1$(mod $N$). It follows that the image of the natural $\mu_N$-morphism $\pi^*(L^{(i)})^{-1}\to\mO_C$ factors through $\mO_C(-bp)$ (and not $\mO_C(-(b+1)p)$). Hence, the natural inclusion $\pi^*(L^{(i)})^{-1}(nq_1)\hookrightarrow\mO_C(mp)$ exists if and only if $nN-b\le m$, i.e. $n\leq\frac{m+b}{N}$. As $b=ia_1-\lfloor\frac{ia_1}{N} \rfloor N$, the desired result follows.
\end{proof}

\subsection{Reduced divisors on a hyperelliptic curve}\label{sec:reduced_divisors}
One result that we will frequently use hereafter is the unique \textit{reduced} presentation of divisor classes in a hyperelliptic Jacobian due to Mumford. This should not be confused with the usual notion of reducedness for a scheme. 
\begin{defn}\label{defn:reduced_div}
Let $B$ be a hyperelliptic curve of genus $\gamma$, and fix a Weierstrass point $q\in B$. A divisor $D$ of degree zero on $B$ is \textit{reduced} (resp., \textit{semi-reduced}) with respect to $q$ whenever the following conditions (resp., the first two listed) are satisfied: 
\begin{enumerate}
    \item $D=\sum_{i=1}^n p_i-nq$, where $p_i$ are all distinct from $q$.
    \item For every $p\in B$, $D$ is not supported simultaneously at $p$ and $\iota(p)$, where $\iota$ is the hyperelliptic involution.
    \item $n\le\gamma$. 
\end{enumerate}
\end{defn}

\begin{lem}[\cite{mumford1984tata}]\label{lem:mumford}
Let $B$ be a hyperelliptic curve of genus $\gamma$, and let $q \in B$ be a Weierstrass point. Each divisor class in $\Jac(B)$ contains a unique representative that is reduced with respect to $q$. 
\end{lem}

We shall use the uniqueness of reduced representatives in the following form:
\begin{cor}\label{cor:reduced}
Every effective divisor $D=\sum_{i=1}^k p_i$ on $B$ for which $D-kq$ is reduced satisfies $h^0(\mO_B(D))=1$ and $h^0(\mO_B(D-q))=0$.   
\end{cor}
\begin{proof}
For the sake of argument, suppose $h^0(\mO_B(D))>1$; then there exists a non-constant function $f$ whose pole divisor is bounded above by $D$. Up to reducing $k$, we may assume the pole divisor of $f$ is precisely $D$. Thus, $D\sim D^{\pr}$, where $D^{\pr}$ is also effective. Displacing multiples of $q$ to the left-hand side of the linear equivalence if necessary (and then subtracting appropriate multiples of $q$ on both sides), we obtain two distinct reduced divisors that are linearly equivalent, a contradiction. The same argument also yields the second assertion.   
\end{proof}

Following Mumford, every semi-reduced divisor on an odd hyperelliptic curve may be presented in a unique way by a pair of univariate polynomials:
\begin{lem}\label{lem:Mumford_presentation}
Let $B$ be a genus $\ga$ hyperelliptic curve with equation $y^2=f(x)$, where $f(x)$ is a monic and separable polynomial of degree $2\ga+1$. 
Let $D=\sum_{i=1}^sp_i-sq$ be a divisor on $B$ that is semi-reduced with respect to $q$.\footnote{We always assume in practice that $q$ is the unique point over $\infty$.} Then there exists a unique pair of polynomials $u(x),v(x)\in k[x]$ satisfying the following conditions:
\begin{enumerate}
    \item $u$ is monic and $\deg u(x)=s$;
    \item $\deg v(x)< \deg u(x)$;
    \item $u(x)| (f(x)-v^2(x))$; and
    \item $p_i=(a_i,v(a_i))$, where $a_1\hh a_s$ are the roots of $u(x)$.  
\end{enumerate}
\end{lem}

On the other hand, Cantor's algorithm (\cite{cantor1987computing}) explains how to compute the {\it reduced} presentation of a semi-reduced divisor:

\begin{lem}\label{lem:Cantor_reduction}
Keeping the same setup as in Lemma~\ref{lem:Mumford_presentation}, suppose that $(u(x),v(x))$ represents the semi-reduced divisor $D$. Set $u_1(x)=\frac{f(x)-v^2(x)}{u(x)}$, renormalizing the highest-order coefficient of $u_1(x)$ to obtain a monic polynomial if necessary; and choose $v_1(x)\equiv -v(x) \, (\text{mod }u_1(x))$ such that $\deg v_1<\deg u_1$. Then $D^{\pr}=\sum_j (p^{\pr}_j-q)\sim D$, where $p^{\pr}_j=(b_j,v_1(b_j))$ and the $b_j$ range over all roots of $u_1(x)$. Moreover, $D^{\pr}$ is semi-reduced with respect to $q$.  
\end{lem}

Notice that $\deg u_1(x)=\max\{2\ga+1,2\deg v(x)\}-\deg u(x)<\deg u(x)$ and that $v_1(x)$ may be chosen so that $\deg v_1(x)<\deg u_1(x)$. Moreover, $u_1| (f-v^2_1)$. After iterating Cantor's algorithm finitely many times, we arrive at a pair $(u_n(x),v_n(x))$ that represents a $q$-reduced divisor. In particular, Cantor's algorithm terminates when $\deg u_i\le \ga$.

\medskip
Now say that applying Cantor's algorithm $T$ times to a Mumford presentation $(u_0(x),v_0(x))$ of a divisor that is semi-reduced with respect to $q$ yields a pair
$(u_T(x),v_T(x))$ that represents a $q$-reduced divisor. Then \[u_{i}(x)\sim \frac{f(x)-v_{i-1}^2(x)}{u_{i-1}(x)}\]
for every $i=1\hh T$.\footnote{Here we use $\sim$ to denote equality up to a scalar.}
In particular, we have
\[u_{T}(x)\sim\frac{f(x)-v_{T-1}^2(x)}{u_{T-1}(x)}\sim \frac{f(x)-v_{T-1}^2(x)}{f(x)-v_{T-2}^2(x)}u_{T-2}(x).\]
By iteratively replacing $u_{i}(x)$ by $\frac{f(x)-v_{i-1}^2(x)}{u_{i-1}(x)}$, we see that
\begin{equation}\label{eqn:cantor_reduction}
\begin{split}
u_T(x)\prod_{k=1}^{{T\over 2}}(f(x)-v_{2k-2}^2(x))&\sim u_0(x)\prod_{k=1}^{{T\over 2}}(f(x)-v_{2k-1}^2(x)) \text{ when $T$ is even; and}\\
u_T(x)u_0(x)\prod_{k=1}^{\lfloor{T\over 2}\rfloor}(f(x)-v_{2k-1}^2(x))&\sim \prod_{k=0}^{\lfloor{T\over 2}\rfloor}(f(x)-v_{2k}^2(x)) \text{ when $T$ is odd}.
\end{split}
\end{equation}

We will use Cantor's algorithm to analyze the existence of torsion points of prescribed order on $B$, as follows. Suppose $N\cd (p-q)\sim 0$, where $N\ge 2\ga +1$; and let $p=(a,b)$ be a point for which $f(a) \neq 0$. Starting with the pair $(u_0(x),v_0(x))=((x-a)^N,v(x))$ that represents $N\cd (p-q)$ and applying the algorithm repeatedly, we eventually arrive at $(u_T(x),v_T(x))=(1,0)$. Indeed, if $u_T(x)$ is not constant, then $(u_T(x),v_T(x))$ represents a \textit{nonzero} divisor of the form $\sum ((a_i,b_i)-q)$, where the $a_i$ range over all roots of $u_T(x)$ and $b_i=v_T(a_i)$, which cannot be the reduced presentation of the trivial divisor. 

\medskip
When $(a,b)$ is torsion of order $N\ge 2\ga+1$, the equations \eqref{eqn:cantor_reduction} specialize to
\begin{equation}\label{eqn:even_tor}
(x-a)^N\prod_{k=1}^{{T\over 2}}(f(x)-v_{2k-1}^2(x))\sim\prod_{k=0}^{{T\over 2}-1}(f(x)-v_{2k}^2(x))
\end{equation}
and 
\begin{equation}\label{eqn:odd_tor}
(x-a)^N\prod_{k=1}^{\lfloor{T\over 2}\rfloor}(f(x)-v_{2k-1}^2(x))\sim \prod_{k=0}^{\lfloor{T\over 2}\rfloor}(f(x)-v_{2k}^2(x))
\end{equation}
respectively.
The degrees of the $v_i$ satisfy recursively-defined inequalities; namely, we have $\deg v_i<\sum_{k=0}^{i}(-1)^{i-k}R_k$, where the integers $R_k$ are prescribed by 
$R_k=\max\{2\ga+1,2\deg v_{k-1}\}$, subject to $R_0=N$.

\begin{lem}
Suppose that $p-q$ is an $N$-torsion, $N\ge 2\ga+1$. Starting with the Mumford presentation $((x-a)^N,v(x))$ of $N\cd (p-q)$, Cantor's algorithm terminates in at most $T$ steps, where $T=
\lceil\frac{N-2\ga}{2}\rceil$. 
\end{lem}

\begin{proof}
It suffices to see that $\deg u_i\equiv \deg u_{i+1}$(mod 2) whenever $2\deg v_i(x)>2\ga+ 1$; otherwise the algorithm terminates after the subsequent step.  
\end{proof}

\begin{cor}\label{cor:small_torsion} Suppose $\ga \geq 2$, and that $f(x)$ is a monic separable polynomial of degree $2\ga+1$.
The hyperelliptic curve 
$y^2=f(x)$ admits an $N$-torsion point with 
\begin{enumerate}
    \item $N=2\ga+1$ if and only if $f(x)=(x-a)^{2\ga+1}+v^2_0(x)$ with $\deg v_0(x) \le \ga$;
    \item $N=2\ga+2$ if and only if $f(x)=(x-a)^{2\ga+2}+v^2_0(x)$ with $\deg v_0(x) =\ga+1$;
    \item $N=2\ga+3$ if and only if $f(x)=\frac{v_0^2(x)-(x-a)^{2\ga+3}\cd v_1^2(x)}{1-(x-a)^{2\ga+3}}$ with $\deg v_0(x) =2\ga+2$ and $\deg v_1(x)\le\ga$.
\end{enumerate}
\end{cor}
\begin{proof}
In the first two cases, the algorithm terminates in one step; while in the last case, it must terminate in two steps. Applying Equations \ref{eqn:even_tor} and \ref{eqn:odd_tor}, respectively, we derive the shape of $f(x)$ in each of these three cases. 

\medskip
More precisely, Cantor's algorithm implies that $N\cd(p-q)\sim 0$ in each of these cases. On the other hand, the fact that $B$ has no $N^{\pr}$-torsion points for any $3\le N^{\pr}\le 2\ga$ (\cite[Corollary 2.8]{zarhin2019division}) implies 
$p-q$ cannot have order smaller than $N$.  
\end{proof}

\begin{rem}
The case of $(2\ga+1)$-torsion points is covered by \cite[Theorem 1]{bekker2020torsion}.
\end{rem}
Whenever $N\geq 2\ga+4$, it is no longer possible to give a uniform characterization of odd-degree hyperelliptic curves with $N$-torsion points, as \textit{a priori} there are various 
distinct ways in which the reduction algorithm might terminate. We nevertheless have the following existence result:
\begin{ex}\label{ex:torsion_1}
Let $1\le s\le\ga$, and set $f(x)=\frac{-1}{2}[(x^{\ga+s}-x^{\ga-s+1}+1)^2-x^{2(\ga+s)}]$, and let $p=(0,\sqrt{\frac{-1}{2}})$. Then, $f(x)$ is a monic, separable polynomial of degree $2\ga+1$, and $(p-q)$ is a $(2\ga+2s)$-torsion point of the hyperelliptic curve $y^2=f(x)$. 

\medskip
Indeed, the fact that $f$ is monic is clear. 
On the other hand, to show that
$f(x)=\frac{1}{2}(2x^{\ga+s}-x^{\ga-s+1}+1)(x^{\ga-s+1}-1)$ is separable it suffices to show that $g(x)=2x^{\ga+s}-x^{\ga-s+1}+1$ is separable, as $f$ clearly does not vanish at $(\ga-s+1)$-th roots of unity; and the separability of $g(x)$ may be checked directly.\footnote{Indeed, $g(x)=g^{\pr}(x)=0$ would imply that $x^{\ga-s+1}=\frac{\ga+s}{2s-1}$ and $x^{2s-1}=\frac{\ga-s+1}{2\ga+2s}$, which is impossible.}  

\medskip
Applying the reduction algorithm to $(x^{2\ga+2s},\sqrt{\frac{-1}{2}}(x^{\ga+s}-x^{\ga-s+1}+1))$, we see that $(2\ga+2s)\cd(p-q)\sim 0$. Since $\ga+s\le 2\ga$, by \cite[Corollary 2.8]{zarhin2019division} we conclude $(p-q)$ has order $2\ga+2s$.   
\end{ex}

\subsection{Codimension estimates}
\label{subsec:effwt}

Several statistics have been proposed to quantify how special a given semigroup is, with the aim of predicting the dimension of the locus $\mathcal{M}^{\rm S}_{g,1} \subseteq \mathcal{M}_{g,1}$ of marked curves realizing that semigroup (when it is nonempty). Whether a semigroup is realizable or not, these statistics serve a useful function in classifying semigroups by their complexity in various senses. Such statistics include: the \emph{weight} $\operatorname{wt}({\rm S})$, equal to the number of pairs $0 < a < b$ of integers with $a \in {\rm S}$ and $b \not\in {\rm S}$; the \emph{effective weight} $\operatorname{ewt}({\rm S})$, equal to the number of such pairs in which $a$ is a \emph{generator}; the number $\dim T^{1,+}(k[{\rm S}])$ defined and studied in \cite{CFSV} in terms of deformations of monomial curves; and Rim--Vitulli's number $\lambda({\rm S})$, given by the number of gaps $b \not\in {\rm S}$ such that $a+b \in {\rm S}$ for all positive elements $a \in {\rm S}$. All these statistics are combinatorial in nature. They are related by the following chain of inequalities (see \cite[p. 2]{CFSV}), in which $g$ denotes the genus of ${\rm S}$:
\[
g - \lambda({\rm S}) \leq \operatorname{codim} \mathcal{M}^{\rm S}_{g,1} \leq g-\lambda({\rm S}) + \dim T^{1,+}(k[{\rm S}]) \leq \operatorname{ewt}({\rm S}) \leq \operatorname{wt}({\rm S})
\]
Here $\operatorname{codim} \mathcal{M}_{g,1}^{\rm S}$ should be omitted when $\mathcal{M}_{g,1}^{\rm S}$ is empty. 
Each of these inequalities may be strict, and some of the limiting cases in which equality holds are well-understood.

\medskip
As mentioned in the introduction, there are few cases where the emptiness or nonemptiness of $\mathcal{M}^{\rm S}_{g,1}$ may be ascertained via  a combinatorial analysis of ${\rm S}$. One such case is when $\operatorname{ewt}({\rm S}) \leq g-1$, in which case $\mathcal{M}_{g,1}^{\rm S}$ is known to have a component of codimension $\operatorname{ewt}({\rm S})$ \cite{Pfl}. Generally speaking, the larger the 
codimension, 
the more complex the semigroup is. 
We will see that almost all multiplication profiles constructed in this paper produce semigroups ${\rm S}$ with $\operatorname{ewt}({\rm S}) \geq g$ and therefore give realizability results not covered by those of \cite{Pfl}. Furthermore, many of our examples satisfy $\operatorname{ewt}({\rm S}) > 3g-2 = \dim \mathcal{M}_{g,1}$, thereby furnishing a large class of examples where ${\rm S}$ is realizable and $\operatorname{codim}\mathcal{M}^{\rm S}_{g,1} < \operatorname{ewt}({\rm S})$. 
The only other class of such examples of which we are aware are the \emph{Castelnuovo semigroups} analyzed in \cite{PflCast}.


\section{Weierstrass-realizability for $2N$-semigroups}\label{sec:Weierstrass-realizability}
\begin{defn}
A numerical semigroup ${\rm S}$ is \textit{Weierstrass-realizable} whenever there exists a pointed smooth projective curve $(C,p)$ for which ${\rm S}={\rm S}(C,p)$ is the Weierstrass semigroup of $C$ at $p$.    
\end{defn}

In investigating the Weierstrass-realizability of a given semigroup ${\rm S}$, it is often to useful to work modulo a fixed element $M \in {\rm S}$. Hereafter, we will refer any numerical semigroup containing a given positive integer $M$ as an \textit{$M$-semigroup}. In practice, $M$ will usually be the {\it multiplicity} of the semigroup ${\rm S}$, i.e., the smallest nonzero element in ${\rm S}$.

\begin{defn}\label{defn:basis}
Given an $M$-semigroup ${\rm S}$, we set 
\[e_i=e_i({\rm S}):=\min\{s\in {\rm S} \backslash \{0\} \mid s\equiv i(\text{mod }M)\}.\] 
The \textit{standard basis} modulo $M$ of ${\rm S}$ is the $M$-dimensional vector $(M, e_1\hh e_{M-1})$.
\end{defn}


\noindent \textbf{Our strategy.} We will certify the Weierstrass-realizability of prescribed $2N$-semigroups ${\rm S}$ by constructing cyclic covers of degree-$N$ of hyperelliptic curves and computing the Weierstrass semigroups at totally ramified points lying over Weierstrass points of the targets. 
We focus on cyclic covers of the following form.
\begin{si}\label{si:cyclic_hyperell}
Let $\pi:C=\Sp(\bigoplus_{i=0}^{N-1}(L^{(i)})^{-1})\to B$ be a cyclic Galois cover of degree $N$ between smooth projective curves, where $B$ is a hyperelliptic curve of genus $\gamma$ and $g(C)=g$. Here $L$ is a line bundle on $B$ for which $L^{\ot N}\cong \mO_B(D)$ (for some fixed choice of isomorphism) where $D$ is an effective divisor of the form $aq+\sum a_iq_i$, and $L^{(i)}=L^i(-[\frac{iD}{N}])$. Fix $p\in C$ such that $\pi(p)=q$, where $q$ is a Weierstrass point of $B$ and $\pi$ is totally ramified at $p$. 

\end{si}

According to Proposition \ref{prop:push}, a given integer $m$ belongs to ${\rm S}(C,p)$ if and only if
\begin{equation}\label{eqn:in_H(p)}
\sum_{i=0}^{N-1}h^0\bigg(\bigg(L^{(i)}\bigg)^{-1}\bigg(\bigg\lfloor\frac{m}{N}+\bigg\lag \frac{ia}{N}\bigg\rag\bigg\rfloor q\bigg)\bigg)\neq   \sum_{i=0}^{N-1}h^0\bigg(\bigg(L^{(i)}\bigg)^{-1}\bigg(\bigg\lfloor\frac{m-1}{N}+\bigg\lag \frac{ia}{N}\bigg\rag\bigg\rfloor q\bigg)\bigg). 
\end{equation}
The inequality holds if and only if there exists at least one $i$ such that $s(m,i):=\frac{m}{N}+\lag \frac{ia}{N}\rag$ is an integer and $h^0((L^{(i)})^{-1}(s(m,i)q))\neq h^0((L^{(i)})^{-1}((s(m,i)-1)q))$. 

\medskip
By definition, $s(m,i)\in\ZZ$ if and only if $m+ia\equiv 0(\text{mod }N)$. Since $\pi$ is totally ramified at $p$, by Theorem \ref{thm:foundation} we have $\gcd(a,N)=1$, and hence there is \textit{precisely} one $i$ such that $s(m,i)$ is an integer. In this case, we shall write $s_m$ for $s(m,i)$.

\begin{defn}\label{defn:Weierstrass_set}
Given any line bundle $L$ on $C$, we set
\[{\rm S}_L(q):=\{n\in \ZZ\mid h^0(L(nq))\neq h^0(L((n-1)q)) \}\]
and refer to it the \textit{Weierstrass set}
at $q$ of $L$.
\end{defn}
\medskip
Given a $2N$-semigroup ${\rm S}$ with standard basis $(2N,e_1\hh e_{2N-1})$, our goal is to decide whether ${\rm S}={\rm S}(C,p)$, for some pair $(C,p)$. To do so, we will adapt the inequality \eqref{eqn:in_H(p)} to the case of $m=e_i$. It is worth noting that specifying a standard basis $(2N,e_1\hh e_{2N-1})$ is equivalent to fixing a {\it gap vector} $(0,v_1\hh v_{2N-1})$, in which $v_i$ denotes the number of elements in $\mb{N} \setminus {\rm S}$ with $2N$-residue equal to $i$, for every $i=1\hh 2N-1$. Indeed, since $e_i$ is a minimal representative of the $i$-th residue class, we have $v_i= \lfloor \frac{e_i}{2N} \rfloor$. 

\medskip 
The upshot of the preceding discussion is the following criterion for the Weierstrass-realizability of $2N$-semigroups:
\begin{lem}\label{lem:tautological}
Let ${\rm S}$ be a $2N$-semigroup with standard basis $(2N,e_1\hh e_{2N-1})$. Then ${\rm S}$ is Weierstrass-realizable by a pointed cyclic cover $\pi:(C,p)\to (B,q)$ as in  Situation \ref{si:cyclic_hyperell} if and only if there exists an invertible sheaf $L$ on $B$ and an effective divisor $D=aq+\sum a_kq_k$ for which $L^{\ot N}\cong\mO_B(D)$, $\gcd(a,N)=1$ and
\[h^0((L^{(i)})^{-1}(s_{\ov{e_j}}q))=1 \text{ and } h^0((L^{(i)})^{-1}((s_{\ov{e_j}}-1)q))=0\]
for $j=1,\dots,N-1$, where $\ov{e_j}=\min\{e_j,e_{j+N}\}$, $\ov{e_j}+ia\equiv 0$(mod $N$) and $L^{(i)}=L^i(-[\frac{iD}{N}])$; moreover, $\max\{e_j,e_{j+N}\}=\min\{n\in {\rm S}_{(L^{(i)})^{-1}}(q)\mid n-s_{\ov{e_j}}\equiv 1 (\text{mod }2)\}$.
\end{lem}

\begin{proof}[Proof of Lemma \ref{lem:tautological}]
We have ${\rm S}={\rm S}(C,p)$ if and only if $e_j=\min\{n\in {\rm S}(C,p)\mid n\equiv j(\text{mod }2N)\}$. This in turn is equivalent to $s_{\ov{e_j}}$ being the minimal element of ${\rm S}_{(L^{(i)})^{-1}}(q)$, which is cohomologically characterized by the given conditions. Notice also that $s_{e_j}=\frac{e_j}{N}+\lag \frac{ia}{N}\rag$ and $s_{e_{j+N}}=\frac{e_{j+N}}{N}+\lag \frac{ia}{N}\rag$ are integers of different parity, which yields the last item.  
\end{proof}
It is easy to see that if $L^{\ot N}\cong\mO_B(D)$ and $D$ is effective on $B$, then $\deg(L^{-1})<0$ and hence that ${\rm S}_{L^{-n}}(q)\subset \ZZ_{>0}$ for every $n\ge 1$.\footnote{In fact, it is also true that ${\rm S}_{(L^{(i)})^{-1}}(q)\subset \ZZ_{>0}$, as $\deg((L^{(i)})^{-1})=-i\cd\deg(L)+\deg\bigg(\mO_B\bigg(\bigg[\frac{iD}{N}\bigg]\bigg)\bigg)$, $\deg(D)=N\cd\deg(L)$, and we have assumed from the outset that $D$ is {\it not} of the form $D=N\cd D^{\pr}$, where $D^{\pr}$ is effective.}

\subsection{Weierstrass-realizability via cyclic covers with reduced branch locus}
A particularly simple class of cyclic covers is comprised of those for which the covering data $(L,D)$ includes a \textit{reduced} effective divisor $D$. Hereafter, we focus on those $2N$-semigroups that are realizable by pointed cyclic covers $\pi:(C,p)\to (B,q)$ constructed out of such covering data. More precisely, we study the Weierstrass-realizability of $2N$-semigroups via degree-$N$ cyclic covers $\pi:C\to B$ of hyperelliptic curves satisfying the following hypotheses:
\begin{si}\label{si:reduced}
\begin{enumerate}
    \item The cover $\pi$ satisfies the conditions of Situation \ref{si:cyclic_hyperell}; and   \item the underlying (branch) divisor $D$ is reduced (and hence $L^{(i)}=L^i$).
\end{enumerate}
\end{si}
In this case, the cover $\pi$ is totally ramified at every point in the ramification locus. Let $g$ (resp., $\ga$) denote the genus of the source (resp., target curve); we have $g=1+N(\gamma-1)+\frac{1}{2}\deg(D)(N-1)$ by the Riemann--Hurwitz formula. {\it Hereafter, we always assume $g\ge (2N-1)\gamma$}. 

\begin{rem}\label{rem:no_redundancy}
A $2N$-semigroup ${\rm S}$ realized by a pointed cyclic cover $\pi: C \ra B$ as in Situation~\ref{si:reduced} with target $B$ of genus $\ga$ is \emph{$(N,\ga)$-hyperelliptic} in the sense of \cite{CarvalhoTorres}; that is, $\mb{N} \setminus {\rm S}$ contains precisely $\ga$ multiples of $N$, all of which are strictly less than $2N\ga$. In particular, semigroups ${\rm S}_1$ and ${\rm S}_2$ associated with distinct pairs $(N_i,\ga_i)$, $i=1,2$ are always distinct.
\end{rem}

\begin{nt}
Given $n \geq 1$, let $a_n(L,q)$ denote the minimal element of ${\rm S}_{L^{-n}}(q)$ strictly greater than zero. 
\end{nt}

\begin{thm}\label{thm:tautological_2}
Let $N$ be a positive integer. Those $2N$-semigroups that arise as  Weierstrass semigroups ${\rm S}= {\rm S}(C,p)$ of totally ramified points $p \in C$ of pointed cyclic covers $\pi:(C=\Sp(\bigoplus_{i=0}^{N-1}(L^{(i)})^{-1}),p)\to (B,q)$ with reduced branch divisors as in Situation \ref{si:reduced} correspond bijectively to tuples of positive integers $(a_1\hh a_{N-1})$ for which there exist a line bundle $L$ and a reduced effective divisor $D$ on $B$ such that 
\begin{enumerate}
    \item $L^{\ot N}\cong \mO_B(D)$; and
    \item $h^0(L^{-j}(a_jq))=1,h^0(L^{-j}((a_j-1)q))=0$, for $j=1, \dots,N-1$.
\end{enumerate}
\end{thm}

\begin{proof}
In the notation of Lemma~\ref{lem:tautological}, we have $s_{\ov{e_j}}=\frac{\ov{e_j}-j}{N}+1$ for every $j=1,\dots,N-1$.
Accordingly, the cohomological condition in Lemma \ref{lem:tautological} simplifies to the requirement that
\begin{equation}
h^0(L^{-(N-j)}(s_{\ov{e_j}}q))=1, h^0(L^{-(N-j)}((s_{\ov{e_j}}-1)q))=0, 1\le j\le N-1.
\end{equation}
Conversely, assume we are given an invertible sheaf $L$ and a reduced effective divisor $D$ for which $L^{\ot N}\cong \mO_B(D)$, 
it is easy to see that $Na_j(L,q)-(N-j)\in {\rm S}(C,p)$ for every $j=1,\dots,N-1$. Suppose now that ${\rm S}$ is a $2N$-semigroup with standard basis $(2N,e_1\hh e_{2N-1})$ and that ${\rm S}={\rm S}(C,p)$. Because of the minimality of the integers $a_j(L,q)$ in their respective Weierstrass sets, we conclude that 
\[
Na_{N-j}(L,q)-(N-j)=\ov{e_j}=\min\{e_j,e_{j+N}\}
\]
or equivalently, that $a_{N-j}(L,q)=s_{\ov{e_j}}$. 
\end{proof}

\begin{nt}
Set $d_{g,\gamma,N}=\frac{(2g-2)-N(2\gamma-2)}{N(N-1)}$.
\end{nt}

Riemann-Hurwitz implies that the cover $\pi:(C,p)\to (B,q)$ is totally ramified over precisely $Nd_{g,\gamma,N}$ distinct points. In particular, we have $\deg(L)=d_{g,\gamma,N}$. Our additional numerical assumption $g\ge (2N-1)\gamma$ ensures that $L^{\ot N}$ is very ample, and thus by Bertini's theorem the associated complete linear series contains a reduced divisor $D$. The upshot is that in constructing cyclic covers 
that realize Weierstrass semigroups, we may focus on verifying the cohomological conditions of Theorem~\ref{thm:tautological_2}. This point is illustrated in the next example. 

\begin{ex}\label{ex:sanity}
Suppose $L=\mO_B(d_{g,\gamma,N}q)$. Note that $(L^{\ot N})(-q)$ is itself very ample; accordingly, we fix a general reduced divisor $D^{\pr}$ in $|(L^{\ot N})(-q)|$ not supported at $q$ and use $D=D^{\pr}+q\in|L^{\ot N}|$ to build a cyclic cover as in Situation \ref{si:reduced}. In so doing we realize the $2N$-semigroup with standard basis satisfying
\[e_0=2N \text{ and } \ov{e_j}=\min\{e_j,e_{N+j}\}=(Nd_{g,\gamma,N}-1)(N-j) \text{ for every } j=1\hh N-1.\]
We will see in Remark \ref{rem:e_from_a} that the numbers $\ov{e_j}$ uniquely determine the standard basis and hence the semigroup.
This semigroup was studied in \cite{CarvalhoTorres}; hereafter we will refer to it as being of {\it Carvalho–Torres type}. 
\end{ex}
Example~\ref{ex:sanity} illustrates our general strategy towards realizing a given numerical semigroup ${\rm S}$: we seek a divisor $Y$ on $(B,q)$ and a reduced effective divisor $D$ on $B$ such that $D\in |NY|$ and $(\mO_B(Y),D)$ provides a cyclic cover $\pi: (C,p)\to (B,q)$ for which ${\rm S}={\rm S}(C,p)$.
We may reformulate our realizability problem as follows: 

\begin{q}\label{q:reformulation}
Let ${\rm S}$ be a $2N$-semigroup with standard basis $(2N,e_1\hh e_{2N-1})$, and set $a_j=\frac{\min\{e_{N-j},e_{2N-j}\}+j}{N}$. 
For which tuples of positive integers $(a_1\hh a_{N-1})$ does there exist a divisor $Y$ of degree $d_{g,\gamma,N}$, not necessarily effective, such that for every $1\le j\le N-1$, we have
\[\begin{cases}
h^0(\mO_B(-(N-j)Y+a_jq))=1 \text{ and}\\
h^0(\mO_B(-(N-j)Y+(a_j-1)q))=0?
\end{cases}\]
\end{q}
Hereafter, we shall refer to the $2N$-semigroup $H$ whose standard basis is determined by $(a_1\hh a_{N-1})$ as above as \textit{the $2N$-semigroup indexed by} $(a_1\hh a_{N-1})$. The fact that $(a_1\hh a_{N-1})$ uniquely determines the standard basis is not obvious, but it is explained in Remark \ref{rem:e_from_a}.

\subsection{Restrictions on Weierstrass semigroups arising from cyclic covers}\label{sec:restrictions_on_semigroups}
Let $\mc{L}(a,b):= L^{\otimes -a}(b \cdot q)$. We begin by explicitly relating the cohomology of $\mc{O}(m \cdot p)$ on $C$ to that of $\mc{L}(a,\lfloor \frac{m+a}{N} \rfloor)$ on $B$. To this end, note first that
\[
h^0(C,m \cdot p)- h^0(C,(m-1) \cdot p)= h^0\bigg(B,\mc{L}\bigg(a,\bigg\lfloor \frac{m+a+1}{N} \bigg\rfloor\bigg)\bigg)- h^0\bigg(B,\mc{L}\bigg(a,\bigg\lfloor \frac{m+a}{N} \bigg\rfloor\bigg)\bigg)
\]
precisely when $N$ divides $(m+a+1)$. 
It follows that
\[
h^0(C,(Nb-a)p)- h^0(C,(Nb-a-1)p)=h^0(B,\mc{L}(a,b))- h^0(B,\mc{L}(a,b-1))
\]
and, consequently, that
\[
h^0(B,\mc{L}(a,b))- h^0(B,\mc{L}(a,b-1))= \delta(Nb-a \in {\rm S}(C,p))
\]
where $\delta(\cdot)$ returns 1 precisely when the parenthesized clause is true, and returns 0 otherwise. As $h^0(B,\mc{L}(a,-1))=0$, we deduce that
\begin{equation}\label{h0B_formula}
\begin{split}
h^0(B,\mc{L}(a,b))&= \sum_{i=0}^n (h^0(B,\mc{L}(a,i))- h^0(B,\mc{L}(a,i-1))) \\
&= \#\{m \in {\rm S}(C,p): m \leq Nb-a \text{ and } m \equiv -a \text{ (mod N)}\}.
\end{split}
\end{equation}
The following result shows that \eqref{h0B_formula} uniquely determines the {\it gap vector} computing the number of elements of $\mb{N} \setminus {\rm S}(C,p)$ in each $N$-residue class between 0 and $N-1$.
\begin{thm}\label{gap_vector_thm}
We have
\[
\#\{\text{gaps } \equiv -a \text{ (mod }N)\}= \ga-1 + \delta(a=0)+ ad_{g,\gamma,N}
\]
for every $a=0,1,\dots,N-1$.
\end{thm}

\begin{proof}
Riemann--Roch implies that 
\begin{equation}\label{RR}
h^0(B,\mc{L}(a,b))= b+ 1- \gamma- a d_{g,\gamma,N}.
\end{equation}
whenever $b \gg 0$. On the other hand, \eqref{h0B_formula} implies that
{\small
\[
    \begin{split}
        h^0(B, \mc{L}(a,b))&= \#\{m \leq Nb-a: m \geq 0 \text{ and } m \equiv -a \text{ (mod N)}\}- \#\{\text{gaps }\equiv -a \text{ (mod N)}\} \\
        &= \bigg\lfloor \frac{Nb-a}{N} \bigg\rfloor+1- (\#\text{gaps }\equiv -a \text{ (mod N)} \\
        &= b+ \bigg\lfloor \frac{-a}{N} \bigg\rfloor+1- (\#\text{gaps }\equiv -a \text{ (mod N)}.
    \end{split}
    \]
}
Comparing the latter equality with \eqref{RR} yields the desired result.
\end{proof}

Additional restrictions on $(a_1,\dots,a_{N-1})$ arise from the additive structure of ${\rm S}$; indeed, according to \cite[(1.2)]{morrison1986galois}, for every $i,j \in \{0,1,\dots,2N-1\}$ we have
\begin{equation}\label{MP_inequalities}
v_i+v_j \geq v_{i+j} \text{ if } i+j<2N, \text{ while } 1+v_i+v_j \geq v_{i+j-2N} \text { if } i+j >2N
\end{equation}
where $v_i$ denotes the cardinality of the set of gaps $\{\ell \in \mb{N} \setminus {\rm S}: \ell \equiv i \text{ (mod 2N)}\}$ with $2N$-residues equal to $i$. Here $v_i=\lfloor \frac{e_i}{2N} \rfloor$. We will refer to the inequalities \eqref{MP_inequalities} as the {\it Morrison--Pinkham inequalities}; they characterize those vectors $(e_0,\dots,e_{2N-1})$ that arise as standard bases of numerical semigroups (and in particular, are independent of conditions arising from the construction of cyclic covers). A useful reformulation of \eqref{MP_inequalities} is the requirement that $(e_0,\dots,e_{2N-1})$ satisfy
\begin{equation}\label{MP_inequalities_alt}
e_j+e_k \geq e_{(j+k) \text{ mod 2N}}
\end{equation}
for every $j,k \in \{0,1,\dots,2N-1\}$.
On the other hand, Theorem~\ref{gap_vector_thm} implies that standard bases $(e_0,\dots,e_{2N-1})$ of $2N$-semigroups that arise from our cyclic covers built out of reduced branch divisors as in Situation \ref{si:reduced} satisfy
\begin{equation}\label{cyclic_cover_standard_basis_restriction}
e_{N-i}+ e_{2N-i}= 2iNd_{g,\ga,N}+ (2\ga+1)N-2i
\end{equation}
for every $i=1,\dots,N-1$. Applying the Morrison--Pinkham inequalities \eqref{MP_inequalities_alt} with $j=N, k=N-i$ in tandem with the cyclic cover restriction \eqref{cyclic_cover_standard_basis_restriction}, we deduce that
\[
iNd_{g,\ga,N}-i \leq e_{N-i}, e_{2N-i} \leq iNd_{g,\ga,N}+ (2\ga+1)N- i
\]
and consequently that
\[
iNd_{g,\ga,N} \leq \min(e_{N-i}, e_{2N-i})+i \leq iNd_{g,\ga,N}+ \ga
\]
for every $i=1,\dots,N-1$. In other words, $a_i \in [id_{g,\ga,N},id_{g,\ga,N}+\ga]$ for every $i$.

\begin{rem} \label{rem:e_from_a}
    Equation \eqref{cyclic_cover_standard_basis_restriction}
    shows that the standard basis $(e_0, \cdots, e_{2N-1})$ is determined completely by $g,\gamma,N$ and only $N-1$ numbers $\min \{ e_{N-j}, e_{2N-j} \}$, for $j=1,\cdots,N$. Indeed, knowing the minimum and the sum determines the \emph{set} $\{e_{N-j}, e_{2N-j} \}$, and the residues modulo $2N$ determine which is which, via
    \[ 
    \min \{e_{N-j}, e_{2N-j} \} = \begin{cases}
        e_{N-j} & \mathrm{if}\, a_j\, \mbox{is odd,}\\
        e_{2N-j} & \mathrm{if}\, a_j\, \mbox{ is even.}
    \end{cases}
    \]
    Here 
    $a_j = \frac{ \min\{a_{N-j},a_{2N-j} \} + j}{N}$ as in
    Question \ref{q:reformulation}.
    This is why the numbers $a_j$ defined uniquely determine the standard basis, and therefore the phrase ``the $2N$-semigroup indexed by $(a_1, \cdots, a_{N-1})$'' is well-defined.
\end{rem}

\subsection{Weierstrass-realizability for feasible $2N$-semigroups}\label{sec:weierstrass-realizability_for_feasibles}
\begin{defn}\label{defn:feasible}
For $N\ge 2$, we set $\ti{F}(N)=\Big(\prod_{j=1}^{N-1}[jd_{g,\gamma,N},jd_{g,\gamma,N}+\gamma]\Big)\cap\ZZ^{N-1}$. Similarly we set
\[F(N)=\{(a_1\hh a_{N-1})\in \ti{F}(N)\mid a_{i+j}\le a_{i}+a_{j} \text{ for every } 1\le i,j,i+j\le N-1\}.\]
We refer to $F(N)$ as the $N$-th {\it feasible set}.
\end{defn}

\begin{ex}\label{ex:N2}
Say $N=2$. The feasible set $F(2)$ is an interval of length $\ga$, and each of the $\ga+1$ possible values $a_1=d_{g,\gamma,2}+j$, $j=0,\dots,\ga$ indexes a $4$-semigroup that is realizable as the Weierstrass semigroup of a (totally ramified point of a) double cover of a hyperelliptic curve. Indeed, according to Corollary~\ref{cor:reduced}, the existence of a line bundle $L$ that satisfies the cohomological conditions of Theorem~\ref{thm:tautological_2} is a straightforward consequence of the existence of $q$-reduced divisors in every divisor class of any given pointed hyperelliptic curve $(B,q)$, which in turn is the content of Lemma~\ref{lem:mumford}.
\end{ex}

\begin{ex}\label{ex:N3} Say $N=3$. This time, we have
\begin{equation}\label{eqn:feasible_set}
F(3):=\{(a_1,a_2)\in ([d_{g,\ga,3},d_{g,\ga,3}+\gamma]\times [2d_{g,\ga,3},2d_{g,\ga,3}+\gamma])\cap\ZZ^2\mid a_2\le 2a_1\}
\end{equation}
and
for each $j=1,2$, the cohomological conditions \[h^0(\mO_B(-jY+a_jq))=1,
h^0(\mO_B(-jY+(a_j-1)q))=0\] of Theorem~\ref{thm:tautological_2}
determine an (at-most) unique effective divisor $E_j\sim a_jq-jY$ on $B$, which may or may not exist. Note that $\deg(E_j)=a_j-j\cd\deg(Y)=a_j-jd_{g,\ga,3}$. 
\end{ex}
With the two preceding examples in mind, we reformulate our realizability criterion once again:

\begin{prop}\label{prop:reduced_pair}
The $2N$-semigroup indexed by $(a_1\hh a_{N-1})$ is Weierstrass-realizable via a cyclic cover as in Situation \ref{si:reduced} if and only if there exists an $(N-1)$-tuple of effective divisors $(E_1\hh E_{N-1})\in \prod_{j=1}^{N-1}\sy^{a_j-jd_{g,\gamma,N}}B$ for which
\begin{enumerate}
    \item $E_j-(a_j-jd_{g,\gamma,N})q$ is $q$-reduced in the sense of Definition~\ref{defn:reduced_div}; and
    \item $j(E_{1}-a_{1}q)\sim E_{j}-a_{j}q$
\end{enumerate}
for every $j=1,\dots, N-1$.
\end{prop}
\begin{proof}
Given $E_1\hh E_{N-1}$ satisfying (1) and (2), let $L=\mO_B(-E_{1}+a_{1}q)$; it then follows from (3) that $L^{j}\cong \mO_B(a_{j}q-E_{j})$. It also follows from (1) and Corollary \ref{cor:reduced} that 
\[h^0(L^{-j}(a_{j}q))=h^0(\mO_B(E_{j}))=1 \text{ and }h^0(L^{-j}((a_{j}-1)q))=h^0(\mO_B(E_{j}-q))=0.\]
Recall that we assume $g$ is large with respect to $\gamma$, so that a reduced representative $D$ in $|L^{\ot N}|$ exists. Applying Theorem~\ref{thm:tautological_2}, we conclude that the $2N$-semigroup indexed by $(a_1,\dots,a_{N-1})$ is Weierstrass-realizable.

\medskip
Conversely, suppose ${\rm S}$ is Weierstrass-realizable via a cyclic cover as in Situation \ref{si:reduced}. According to Theorem~\ref{thm:tautological_2}, we have $h^0(L^{-j}(a_{j}q))=1$ and $h^0(L^{-j}((a_{j}-1)q))=0$, for $j=1\hh N-1$. This means that there exist effective divisors $E_1\hh E_{N-1}$ for which $\mO_B(E_j)\cong L^{-j}(a_jq)$ for every $j$. In particular, 
$\deg(E_j)=a_j-jd_{g,\gamma,N}\le \gamma$. Consequently $E_j-(a_j-jd_{g,\gamma,N})q$ fails to be reduced if and only if one of the following two situations occurs:
\begin{enumerate}
    \item $E_j$ is supported at $q$; or else
    \item $E_j$ is supported at $p$ and $\iota(p)$, where $\iota$ is the hyperelliptic involution on $B$.
\end{enumerate}
The fact that $h^0(L^{-j}((a_{j}-1)q))=h^0(\mO_B(E_j)(-q))=0$ rules out the first case; while the fact that $h^0(L^{-j}(a_{j}q))=h^0(\mO_B(E_j))=1$ rules out the second. Finally, since $\mO_B(E_j)\cong L^{-j}(a_jq)$ for every $j$, item (2) in the statement of the lemma clearly holds. This concludes the proof.  
\end{proof}
Proposition~\ref{prop:reduced_pair} shows that the problem of constructing pointed cyclic covers of a hyperelliptic curve with prescribed Weierstrass semigroup is fundamentally about the behavior of multiplication maps on a hyperelliptic Jacobian.
\begin{nt}
Let $u_d:B^d\to\Jac(B):\sum p_i\mapsto [\sum p_i-dq]$ be the $d$-th Abel-Jacobi map. Let $\Theta_d$ denote the image of $u_d$, and set $\Theta^0_d=\Theta_d\backslash\Theta_{d-1}$.
\end{nt}
\begin{nt}
Let $m_j:\Jac(B)\to \Jac(B)$ denote the multiplication-by-$j$ map. 
\end{nt}
Proposition~\ref{prop:reduced_pair} establishes that the $2N$-semigroup indexed by $(a_1\hh a_{N-1})$ is realizable via a cover as in Situation \ref{si:reduced} if and only if $\bigcap_{j=1}^{N-1}m_j^{-1}(\Theta^0_{a_j-jd_{g,\gamma,N}})$ is non-empty. 
\begin{rem}
It is worth noting that more refined information is known about set-theoretic intersections of the form $(m_2^{-1}\Theta_d)\cap \Theta_{d^{\pr}}$. For example, when $\gamma>1$, $m_2^{-1}\Theta_1\cap \Theta_1$ consists of the $u_1$-images of all Weierstrass points in $\Jac(B)$, and $m_2^{-1}\Theta_1\cap \Theta_{\gamma-1}\subset\Jac(B)[2]$ (see \cite[Thm 2.5]{zarhin2019division}).
\end{rem}

For the sake of simplicity, hereafter we set $\ep_j=a_j-jd_{g,\ga,N}$ and use the vector $(\ep_1\hh\ep_{N-1})$ to indicate the semigroup previously indexed by $(a_1\hh a_{N-1})$. For future reference, we make a general definition: 
\begin{defn}\label{defn:ep_vec}
Let ${\rm S}$ be an $(N,\ga)$-hyperelliptic semigroup of genus $g\ge(2N-1)\ga$ with standard basis $\{2N,e_1\hh e_{2N-1}\}$, and suppose $d=\frac{(2g-2)-N(2\ga-2)}{N(N-1)}$ is an integer. 
For every $j=1,\dots,N-1$, let $\ep_j=\frac{\min\{e_{N-j},e_{2N-j}\}+j}{N}-j\cd d$; we call $(\ep_j)_{j=1}^{N-1}$ the \emph{$\ep$-vector} of ${\rm S}$.
\end{defn}

Remark \ref{rem:e_from_a} explains why the $\ep$-vector uniquely determines the semigroup, and how to compute the standard basis from the $\ep$-vector.

In these new coordinates, the feasible set $F(N)$ becomes
\begin{equation}\label{eqn:feasible_ep}
    F(N)=\{(\ep_1\hh \ep_{N-1})\in [\ga]^{N-1}\mid \ep_{i+j}\le \ep_i+\ep_j \text{ for every } 1\le i,j,i+j\le N-1\}
\end{equation}
where $[\ga]=\{0,1,\dots,\ga\}$.

\begin{prop}\label{prop:basic}
Suppose the $2N$-semigroup indexed by $(\ep_1\hh \ep_{N-1})$ is realizable via a cyclic cover as in Situation \ref{si:reduced}. Then the following hold:
\begin{enumerate}
    \item The set $\{j\mid \ep_j=0\}$ is either empty or consists of all multiples of a given integer $n$ in $\{1,\hh N-1\}$ and $\ep_{j+n}=\ep_j$ for every $j$. 
    \item When $\gamma>1$ and $\ep_1=1$, we have $\ep_2,\ep_3\neq 1$ unless $(\ep_1,\ep_2,\ep_3)=(1,0,1)$.
    \item When $\gamma>1$ and $\ep_1=1$, we have $\ep_j\neq 0$ for every $j\le 2\ga$ unless $j$ is even and $\ep_2=0$.  
    \item The semigroup indexed by $(\ga\hh\ga)$ is realizable.
    \item If $\ep_i=0$, then $\ep_{j}=\ep_{i-j}$ for every index $j$ strictly less than $i$. 
\end{enumerate}
\end{prop}
\begin{proof}
\begin{enumerate}
    \item Without loss of generality, suppose that the index set $\{j\mid \ep_j=0\}$ is nonempty; let $n$ be the smallest such index. Then $E_1-\ep_1q$ defines an order-$n$ element of $\Jac(B)$, and the conclusion follows from the requirement that $j(E_1-\ep_1q)\sim E_j-\ep_jq$. 
    \item The fact that $m_2^{-1}\Theta_1\cap \Theta_1$ consists exclusively of Weierstrass points implies that $\ep_2 \neq 1$; on the other hand, we have $H^0(\mO(2q+p))=H^0(\mO(2q))$ for every $p \in B$.
    \item Follows from the fact that $B$ has no points of order $j \in \{3\hh 2\ga\}$ (\cite[Cor. 2.8]{zarhin2019division}).
    \item Immediate, as the preimages $m_j^{-1}(\Theta^0_{\gamma})$ are non-empty open subsets of $\Jac(B)$.  
    \item Let $x=E_1-\ep_1 q$. If $\ep_i=0$, then $i\cd x=0$, which in turn implies $j\cd x=-(i-j)\cd x$. The item follows from the fact that every $\Theta_j$ is invariant under multiplication by $(-1)$: the inverse of $\sum_{k=1}^j p_k-jq$ is $\sum_{k=1}^j \iota(p_k)-jq$. 
\end{enumerate}
\end{proof}

\begin{ex}\label{sec:realizability_small_target_genera}
Suppose $\ga=1$. In this case, the coordinates $\ep_i$ of every candidate vector $(\ep_1\hh\ep_{N-1})$ satisfy $\ep_{i}+\ep_j\ge \ep_{i+j}$ for every pair of indices $i,j \leq \frac{N-1}{2}$. Furthermore, we have $m_j^{-1}(\Theta^0_1)=B\backslash B[j]$ and $m_j^{-1}(\Theta_0)=B[j]$. Candidate $\ep$-vectors that are realizable by cyclic covers as in Situation~\ref{si:reduced} are as follows:
\begin{enumerate}
    \item $(\ep_1\hh\ep_{N-1})=\mathbf{0}$. The trivial $\ep$-vector indexes a Carvalho–Torres semigroup, which is realizable; see Example~\ref{ex:sanity}.
    \item For some fixed $i$, we have $\ep_{ki}=0$ for every $k=1,\dots, \lfloor \frac{N-1}{i} \rfloor$; and $\ep_j=1$ for all other indices $j$. To realize the corresponding semigroup by a cyclic cover as in Situation~\ref{si:reduced}, choose $E_1=p-q$, where $p \in B$ is any point (with Abel--Jacobi image) of order $i$.  
    \item $(\ep_1\hh\ep_{N-1})=\mathbf{1}$. Realization follows trivially for topological reasons in this subcase; see Proposition~\ref{prop:basic}(4).
\end{enumerate}
\end{ex}

\subsection{Realizability in small covering degree}\label{sec:realizability_small_covering_degree}
\begin{ex}\label{ex:even_gen}
Whenever $1\le k\le\frac{\gamma}{N-1}$, $\sum_{i=1}^k p_i-kq$ is reduced and no $p_i$ is a Weierstrass point on $B$, the first $(N-1)$ multiples $j(\sum_{i=1}^k p_i-kq)$, $1\le j\le N-1$ are also reduced. 

\medskip
Accordingly, suppose that $1\le k\le\frac{\gamma}{N-1}$, $\sum_{i=1}^k p_i-kq$ is reduced and no $p_i$ is a Weierstrass point, and set $L=\mO_B((d_{g,\gamma,N}+k)q-\sum_{i=1}^k p_i)$. Corollary \ref{cor:reduced} now implies that
\[h^0(L^{-j}((jd_{g,\gamma,N}+jk)q))=1,h^0(L^{-j}((jd_{g,\gamma,N}+jk-1)q))=0\]
for every $j=1,\dots,N-1$. It follows that the $2N$-semigroup with associated $\ep$-vector $(k,2k\hh(N-1)k)$ is realizable, for every $k=1,\dots, \frac{\gamma}{N-1}$.
\end{ex}

\subsection*{N=2} As explained in Example \ref{ex:N2}, the realizability problem is trivial in this case and every $4$-semigroup is realizable. 

\subsection*{N=3}In this case, we obtain the following result:
\begin{thm}\label{thm:realizability}
The 6-semigroup with associated $\ep$-vector $(\ep_1,\ep_2)$ is Weierstrass-realizable by a construction described as in Situation~\ref{si:reduced}
if and only if either 
\begin{enumerate}
    \item $(\ep_1,\ep_2)\in F(3)$ and $\ep_2$ is even; or
    \item $(\ep_1,\ep_2)\in F(3)$, $2\gamma-2\ep_1+1\le \ep_2\le 2\ep_1$, and $\ep_2$ is odd.
\end{enumerate}
\end{thm}

\noindent We split the argument into two cases, depending upon the parity of $a_2$.

\subsection*{Case one: $\ep_2$ is even}
\begin{lem}\label{lem:even}
The 6-semigroups indexed by $(\ep_1=k,\ep_2=2m)\in F(3)$ are all Weierstrass-realizable.
\end{lem}
\begin{proof}
Note that the feasible range of $m$ is $0\le m\le\frac{\gamma}{2}$ and the corresponding feasible range of $k$, given $m$, is $m\le k\le \gamma$.  

\medskip
Now set $Z=\sum_{i=1}^mp_i$ be an effective divisor as in Example~\ref{ex:even_gen} that underlies a cyclic cover realizing the 6-semigroup with $\ep$-vector $(m,2m)$; and suppose that $m<k\le \gamma$. Let $E_1=Z+\sum_{i=1}^{k-m}q_i$, where the $q_i$ are distinct Weierstrass points other than $q$. Note that $E_1-kq$ is a reduced divisor. Set $L=\mO_B((d_{g,\ga,3}+k)q-E_1)$; Corollary \ref{cor:reduced} implies that
\[
h^0(L^{-1}((d_{g,\ga,3}+k)q))=h^0(\mO_B(E_1))=1, h^0(L^{-1}((d_{g,\ga,3}+k-1)q))=0
\]
and also
\[h^0(L^{-2}((2d_{g,\ga,3}+2m)q))=h^0(\mO_B(2Z))=1, h^0(L^{-1}((2d_{g,\ga,3}+2m-1)q))=0.\]
Our proof is now complete. 
\end{proof}

\subsection*{Case two: $\ep_2$ is odd}
When $\ep_2=2m+1$ is odd, we have further restrictions on the feasible range of $(\ep_1,\ep_2)$ beyond what is given in Equation (\ref{eqn:feasible_set}):

\begin{lem}\label{lem:odd_necessary}
Every 6-semigroup with $\ep$-vector $(\ep_1,2m+1)\in F(3)$ that is Weierstrass realizable by a cyclic cover as in Situation~\ref{si:reduced} satisfies $\ep_1\ge \gamma-m$.
\end{lem}

\begin{proof}
According to Proposition~\ref{prop:reduced_pair}, any cyclic cover that realizes $(\ep_1,2m+1)$ is associated with a pair of effective divisors $E_1,E_2$ for which the corresponding degree-0 divisors $E_i-\deg(E_i)\cdot q$, $i=1,2$ are reduced. Moreover, the condition $2[E_1-\ep_1q]\sim E_2-(2m+1)q$ implies that
\[[2\ep_1-(2m+1)]q\sim 2E_1-E_2\sim 2E_1+\iota(E_2)-2(2m+1)q.\]
Thus $[2\ep_1+(2m+1)]q\sim 2E_1+\iota(E_2)$. Since there is no further cancellation on the two sides, we conclude that there exists a rational function on $B$ with (odd) pole order $2\ep_1+(2m+1)$ at $q$. Since $q$ is a Weierstrass point, $2\ep_1+(2m+1)\ge 2\gamma+1$ must hold; unwinding yields $\ep_1\ge \gamma-m$. 
\end{proof}

Next, we shall show that for every $(\ep_1,\ep_2)\in F(3)$ such that $\ep_1\ge \gamma-m$ and $\ep_2=2m+1$, the corresponding 6-semigroup is Weierstrass-realizable via a triple cover described in Situation~\ref{si:reduced}. In other words, the previous lemma characterizes the only obstruction to realizability when $N=3$. We start with an instructive special case.

\begin{lem}\label{lem:b_odd2}
Every semigroup with $\ep$-vector $(\gamma,2m+1)\in F(3)$ is realizable for every $m \geq 0$.
\end{lem}

\begin{proof}
We proceed by induction on $m$. For the base case $m=0$, we use the fact that $m_2^{-1}\Theta_1\cap \Theta_{\gamma-1}\subset\Jac(B)[2]$ (\cite[Thm 2.5]{zarhin2019division}). Given a general point $q^{\pr}\in B$, it follows that there is an effective divisor $E_1$ for which $E_1-\ga q$ is reduced and $[2(E_1-\ga q)]=[q^{\pr}-q]$. In particular, the pair $(E_1,q^{\pr})$ resolves the realizability problem for $(\ep_1,\ep_2)=(\gamma,1)$.

\medskip
Now let $(E_1,E_2)$ be a pair of effective divisors as in Proposition~\ref{prop:reduced_pair} that resolves the realizability problem for $(\ep_1,\ep_2)=(\gamma,2m-1)$; in particular, we have $\deg(E_1)=\gamma$ and $\deg(E_2)=2m-1$.
Given a general point $q^{\pr}\in B$, let $D_2=E_2+2q^{\pr}$. The fact that $q^{\pr}$ is general guarantees that $D_2-(2m+1)q$ is reduced, and trivially we have 
\[D_2-(2m+1)q\sim 2[E_1+q^{\pr}-(\gamma+1)q].\]
As $\deg(E_1+q^{\pr})=\gamma+1$, we have $E_1+q^{\pr}\sim D_1+q$ for some effective divisor $D_1$ of degree $\gamma$. Thus $D_2-(2m+1)q\sim 2(D_1-\gamma q)$.

\medskip
We now check that $D_1$ is not supported at $q$. Indeed, otherwise we would have 
\[E_1\sim D^{\pr}-q^{\pr}+nq\sim D^{\pr}+\iota(q^{\pr})+(n-2)q \]
where $D^{\pr}$ is effective and not supported at $q$, and $n\ge 2$. That is, \[E_1-\gamma q\sim D^{\pr}+\iota(q^{\pr})-[\gamma-(n-2)]q.\]
By uniqueness of $q$-reduced representatives (Lemma \ref{lem:mumford}), the latter equality is only possible if $n=2$ and $E_1=D^{\pr}+\iota(q^{\pr})$, a possibility that is precluded because $q^{\pr}$ is general. 

\medskip
Finally, we check that  $h^0(\mO_B(D_1))=1$. Given the previous paragraph, either $D_1-\gamma q$ is reduced, in which case we are done; or else $D_1$ is simultaneously supported at $p$ and $\iota(p)$ for some $p \in B$. The latter possibility would also violate the uniqueness of $q$-reduced presentation for $E_1-\gamma q$, however, in view of $E_1+q^{\pr}\sim D_1+q$. By Proposition~\ref{prop:reduced_pair}, we are done. 
\end{proof}
In fact, the 
same argument leads to a general solution to the realizability problem for all $(\ep_1,\ep_2)\in F(3)$ for which $\ep_2=2m+1$ is odd and $\ep_1\ge (\gamma-m)$:

\begin{lem}\label{lem:odd_final}
For every $\ep$-vector $(k,2m+1)$ for which $k\ge \gamma-m$ and $2m+1\le\gamma$, the corresponding $6$-semigroup is realizable by a cyclic triple cover as in Situation \ref{si:reduced}. 
\end{lem}

\begin{proof}
According to Proposition~\ref{prop:reduced_pair}, any solution to the semigroup-realizability problem for $(\ep_1,\ep_2)=(k,2m+1)$ is presented by a pair of effective divisors $(E^1_{k,m},E^2_{k,m})$. In Lemma~\ref{lem:b_odd2}, every pair $(E^1_{\gamma,m},E^2_{\gamma,m})$ is such that
\begin{enumerate}
    \item $E^2_{\gamma,0}=q'$, where $q'$ is a general point on $B$; and
    \item $E^1_{\gamma,m}+q\sim E^1_{\gamma,m-1}+q_m$, $E^2_{\gamma,m}=E^2_{\gamma,m-1}+2q_m$, where $q_m$ is general.  
\end{enumerate}
In particular, we have $\supp(E^1_{\gamma,0})-\supp(E^2_{\gamma,0})\neq\emptyset$, as otherwise $(2\gamma-1)(q^{\pr}-q)\sim 0$, which would violate the generality of $q^{\pr}$. In fact $q^{\pr}\notin\supp(E^1_{\gamma,0})$, as otherwise
\[2E^{\pr}+q^{\pr}\sim (2\gamma-1)q\]
where $E^{\pr}=E^1_{\gamma,0}-q^{\pr}$, which would imply the existence of a rational function on $B$ with pole order $2\gamma-1$ at $q$. This is clearly impossible. A similar argument shows that $\supp(E^1_{\gamma,0})-\supp(E^2_{\gamma,0})$ contains no Weierstrass point of $B$.

\medskip
\noindent Now suppose that $p\in \supp(E^1_{\gamma,0})-\supp(E^2_{\gamma,0})$, and set
\[E^1_{\gamma-1,1}=E^1_{\gamma,0}-p,E^2_{\gamma-1,1}=E^2_{\gamma,0}+2\iota(p).\]
Thanks to our choice of $p$, the divisors $E^1_{\gamma-1,1}-(\gamma-1)q$, $E^2_{\gamma-1,1}-3q$ are $q$-reduced, and 
\[2[E^1_{\gamma-1,1}-(\gamma-1)q]\sim E^2_{\gamma,0}-2p+2q\sim E^2_{\gamma-1,1}-3q.\]
By Proposition~\ref{prop:reduced_pair} again, we obtain a solution for the realizability problem for the semigroup whose $\ep$-vector is $(\gamma-1,3)$. Notice that $\supp(E^1_{\gamma,1})-\supp(E^2_{\gamma,1})$ is nonempty and contains no Weierstrass point; so by induction, we obtain solutions for the realizability problem for those 6-semigroups whose $\ep$-vector is $(\gamma-m,2m+1)$ for which $0\le m<\frac{\gamma}{2} $.

\medskip
The general statement now follows from a simple construction. For $m=1,\dots,\frac{\gamma}{2}$, let $(E^1_{\gamma-m,m},E^2_{\gamma-m,m})$ be a solution to the realizability problem for 6-semigroups with $\ep$-vector $(\gamma-m,2m+1)$. For any $m^{\pr}$ between $1$ and $m$, let $q_1\hh q_{m^{\pr}}$ be distinct Weierstrass points on $B$ not in $\supp (E^1_{\gamma-m,m})\cup\{q\}$.\footnote{Such points exist because $\deg(E^1_{\gamma-m,m})=\gamma-m$, and there are $2\ga+2$ distinct Weierstrass points on $B$.}
\,Let $E^1_{\gamma-m+m',m}=E^1_{\gamma-m,m}+\sum_{k=1}^{m^{\pr}}q_k$. It follows easily from our choice of $q_1\hh q_{m^{\pr}}$ that $E^1_{\gamma-m+m^{\pr},m}-(\gamma-m+m^{\pr})q$ is $q$-reduced  and that \[2(E^1_{\gamma-m+m^{\pr},m}-(\gamma-m+m^{\pr})q)\sim 2(E^1_{\gamma-m,m}-(\gamma-m)q)\sim E^2_{\gamma-m,m}-(2m+1)q.\] 
It follows that $(E^1_{\gamma-m+m^{\pr},m},E^2_{\gamma-m,m})$ underlies a cyclic cover that realizes the 6-semigroup whose $\ep$-vector is $(\gamma-m+m',2m+1)$. 
\end{proof}

Unwinding the definitions of our semigroup indexing schemes, the following is a reformulation of Theorem~\ref{thm:realizability} in terms of standard bases.
\begin{thm}\label{thm:realizability3}
    Let ${\rm S}$ be a $(3,\ga)$-hyperelliptic semigroup of genus $g \geq 5\ga$ whose standard basis is $(6,e_1,\dots,e_5)$. Then ${\rm S}$ is Weierstrass-realizable via a cyclic cover as in Situation~\ref{si:reduced} if and only if
    \begin{enumerate}
        \item either $e_4<e_1$; or else
        \item  $e_4>e_1\ge 4g-6\ga+7-2\min\{e_2,e_5\}$.
    \end{enumerate}
\end{thm}
\begin{proof}
    It suffices to check that the given numerical hypotheses are equivalent to those of Theorem~\ref{thm:realizability}. 
    For this purpose, recall that $\ep_2=\frac{\min\{e_1,e_4\}+2}{3}-2d_{g,\ga,3}$, which has the same parity as $\frac{\min\{e_1,e_4\}+2}{3}$. The latter is even if and only if $\min\{e_1,e_4\}=e_4$, i.e. $e_4<e_1$.

\medskip
    \noindent Whenever $\ep_2$ is odd, i.e. whenever $e_1<e_4$, the 
    lower bound on $\ep_2$ in Theorem~\ref{thm:realizability} becomes
    \[2\ga+1-2\bigg(\frac{\min\{e_2,e_5\}+1}{3}-d_{g,\ga,3}\bigg)\le \frac{e_1+2}{3}-2d_{g,\ga,3}.\]
Substituting $d_{g,\ga,3}=\frac{(2g-2)-3(2\ga-2)}{6}$ and simplifying, we obtain \[e_1\ge 4g-6\ga+7-2\min\{e_2,e_5\}.\]
The other half of the inequality is guaranteed by Equation~\eqref{MP_inequalities_alt}, i.e. by the additive structure of ${\rm S}$. 
\end{proof}
With the aid of 
a computer, we found all such numerical semigroups of genus $g \leq 100$. The code may be found at \url{https://github.com/npflueger/cyclicCoverSemigroups} and the list of 
examples may be found at \url{https://github.com/npflueger/cyclicCoverSemigroups/blob/main/listSemigroups-output.txt}.

\section{The multiplication profile of a point on a hyperelliptic Jacobian}\label{sec:mult_profile}
Thus far, we have used three different indexing schemes 
for $2N$-semigroups ${\rm S}$:
\begin{enumerate}
    \item The standard basis: $e_0,e_1\hh e_{2N-1}$;
    \item the normalized standard basis: $a_j:=\frac{\min\{e_{N-j},e_{2N-j}\}+j}{N}$, $j=1\hh N-1$;  
    \item and, whenever ${\rm S}$ is $(N,\ga)$-hyperelliptic of genus $g \geq (2N-1)\ga$, and $d_{g,\gamma,N}=\frac{(2g-2)-N(2\gamma-2)}{N(N-1)}$ is an integer, 
    by $\ep_j:=a_j-j\cd d_{g,\ga,N}$, $j=1\hh N-1$. 
\end{enumerate}
The third item is most relevant when ${\rm S}$ is realized by a degree-$N$ cyclic cover $\pi:C\to B$, where $C$ (resp., $B$) is a curve of genus $g$ (resp., $\ga$). In this case, the sequence $(\ep_1\hh \ep_{N-1})$ arises from the images under multiplication by $m_j$, $j=1,\dots, N-1$ of an element of the Jacobian $\Jac(B)$. The following definition makes this explicit, and is consistent with our usage in the preceding section.
\begin{defn}
Given $x\in\Theta^0_d\subset \Jac(B)$, let $\ep_j$ denote the unique integer for which $m_j(x)\in\Theta^0_{\ep_j}$ for every positive integer $j$. We refer to the integer sequence $m(x):=(\ep_1=d,\ep_2,\ep_3,\hdots)$ 
as the \textit{multiplication profile} of $x$. We also refer to the subsequence $m(x)_{N-1}:=(\ep_j)_{j=1}^{N-1}$ as the $(N-1)$\textit{-th truncated multiplication profile} of $x$.
\end{defn}
\begin{rem}
When $\ga=1$, possible multiplication profiles are as follows: $m(x)=(0,0,\hdots)$ when $x$ is the neutral element; $m(x)=(1,1,\hdots)$ when $x$ is 
non-torsion; and otherwise $m(x)$ is of some finite order $n$, in which case the multiplication profile is determined by its truncation $m_n(x)=(1\hh 1,0)$. Throughout the remainder of this section, we assume that $\ga\ge2$. 
\end{rem}

Hereafter, we will focus on {\it point classes} that arise from Abel-Jacobi maps of degree one.
\begin{nt}
For a fixed point $p\neq q$, let $R_N(p)$ denote the $q$-reduced representative of $[N(p-q)]$. 
\end{nt}

\begin{lem}\label{lem:nece1}
For every index $j$, we have $\ep_j>\ep_{j+1}$ if and only if $\iota(p)\in\supp(R_j(p))$ for every $p\neq q$. 
\end{lem}

\begin{proof}
Note that $R_j(p)\sim R_{j+1}(p)+(\iota(p)-q)$. As $\ep_{j+1}<\ep_N$, it follows that $R_j(p)=R_{j+1}(p)+(\iota(p)-q)$. 
Conversely, suppose $R_j(p)=E+(\iota(p)-q)$, where $E\in\Theta_{\ep_j-1}^0$. Then $E$ is $q$-reduced, and $R_{j+1}(p)=E$.
\end{proof}

\begin{cor}\label{cor:nece2}
We have $|\ep_j-\ep_{j+1}|\le 1$ for every $j$; and $\ep_j\neq \ep_{j+1}$ whenever $\ep_j<\ga$.
\end{cor}

\begin{proof}
According to Lemma~\ref{lem:nece1}, we have $\ep_{j+1}=\ep_j-1$ whenever $\ep_{j+1}<\ep_j$. If $\ep_{j+1}>\ep_j$, then because $R_j(p)\sim R_{j+1}(p)+(\iota(p)-q)$ is not $q$-reduced, we must have $R_j(p)=R_{j+1}(p)-(p-q)$, and hence $\ep_{j+1}=\ep_j+1$. On the other hand, if $\ep_j<\ga$, then either $\iota(p)\in\supp(R_j(p))$ and $\ep_{j+1}<\ep_j$; or else $\iota(p)\notin\supp(R_j(p))$ and $\ep_{j+1}>\ep_j$.
\end{proof}

\begin{lem}\label{lem:nece3}
Fix $p\neq q$. There is no index $j$ for which $\ep_{j+1}=\ep_j+1=\ep_{j+2}+1$, unless $\ep_j=\ep_2=0$.
\end{lem}

\begin{proof}
The equalities $\ep_{j+1}=\ep_j+1=\ep_{j+2}+1$ force $R_{j+1}(p)=R_{j}(p)+(p-q)=R_{j+2}(p)+(\iota(p)-q)$, in which case $p=\iota(p)$ necessarily.
\end{proof}

\noindent In light of the above, we now develop two instructive examples.

\begin{ex}\label{ex1}
Whenever $x=[p-q]$ is a general point of $\Theta_1$, Lemma~\ref{lem:nece3} implies that
\[m(x)_{\ga+1}=(1,2\hh\ga,\ga).\]

\medskip
\noindent On the other hand, whenever $x\in \Theta_1$ is $2$-torsion, we have $m(x)=(1,0,1,0,\hdots)$.

\medskip
\noindent Finally, whenever $x\in \Theta_1$ is $(2\ga+1)$-torsion, we have $m(x)_{2\ga+1}=(1,2\hh\ga,\ga,\ga-1\hh,1,0)$, and $m(x)_{2\ga+1}$ uniquely determines $m(x)$. 

\medskip
\noindent The above gives a complete solution for Weierstrass-realizability by cyclic covers whenever $2\le N\le \ga+2$ and $a_1=1+d_{g,\ga,N}$. 
\end{ex}

\begin{ex}
Let $a$ be a general element of $k^{\times}$, $1\le n\le \ga+1$, and $f(x)=(x-a)^{\ga+n}x^{\ga-n+1}+1$. Working over the hyperelliptic curve $B:y^2=f(x)$, let $p=(a,1)$ and $p^{\pr}=(0,1)$. Note that $\iota(p)=(a,-1)$ and $\iota(p^{\pr})=(0,-1)$.

\medskip
\noindent Now $\dv(y-1)=(\ga+n)p+(\ga-n+1)p^{\pr}-(2\ga+1)q$. In particular, we have $[(\ga+n)(p-q)]=[(\ga-n+1)(\iota(p^{\pr})-q)]$, and it follows that  
\[
[(\ga+m)(p-q)]=[(\ga-n+1)(\iota(p^{\pr})-q)+(m-n)(p-q)].
\]
Note that $(\ga-n+1)(\iota(p^{\pr})-q)+(m-n)(p-q)$ is $q$-reduced for every $m=n \hh 2n-1$. We claim that
\[m([p-q])_{\ga+2n-1}=(1,2\hh\ga,\ga,\ga-1\hh\ga-n+1,\ga-n+2\hh \ga).\] 
Indeed, given our previous discussion, it only remains to show that $\ep_{\ga+1+j}=\ga-j$ for every $j=1\hh n-2$. This follows, in turn, from the fact that $|\ep_k-\ep_{k+1}|\le 1$: letting $R_k$ denote the $q$-reduced presentation of $m_k([p-q])$, we have $R_k\sim R_{k+1}+\iota(p)-q$.  

\medskip
\noindent Notice that the case in which $[p-q]$ is $(2\ga+1)$-torsion is covered by this example.
\end{ex}

\subsection{Multiplication profiles of point classes}
We begin 
by recapitulating what we know about the
multiplication profile of an element of $\Theta_1^{0}$.
\begin{prop}\label{prop:profile_conditions}
The multiplication profile $(\ep_i)_{i=1}^{\infty}$ of an element of $[p-q]\in\Theta_1^0(B)\subset \Jac(B)$, for some hyperelliptic curve $B$ of genus $\ga$ satisfies the following conditions:
\begin{enumerate}
    \item $\ep_1=1$, and $0\le \ep_j\le\ga$ for every $j$;
    \item $\ep_i+\ep_j\ge \ep_{i+j}$;
    \item $|\ep_j-\ep_{j+1}|\le 1$ and if $\ep_j<\ga$, then $\ep_j\neq \ep_{j+1}$;
    \item\label{it1} there is no $j$ for which $\ep_{j+1}=\ep_j+1=\ep_{j+2}+1$, unless $\ep_j=\ep_2=0$;
    \item the first index $N$ for which $\ep_N=0$ is either $N=2$ or else $N\ge 2\ga+1$, in which case $\ep_{sN+r}=\ep_{r}$ and $\ep_{N-r}=\ep_{r}$ for every $1\le r<N$ and $s\ge 0$.
\end{enumerate}
\end{prop}
\begin{proof}
Items (1), (2) follow from the definition of $\ep_j$; (3) is the content of Corollary~\ref{cor:nece2}; (4) is Lemma \ref{lem:nece3}; while (5) follows from the definition of torsion order and the non-existence of torsion points of small orders (\cite[Corollary 2.8]{zarhin2019division}).
\end{proof}

\begin{conj}\label{conj:profile}
Any sequence $(\ep_i)_{i=1}^{\infty}$ of non-negative integers that satisfies the conditions of Proposition~\ref{prop:profile_conditions} is realized as the multiplication profile $(\ep_i)_{i=1}^{\infty}$ of some $[p-q]\in\Theta_1^0(B)\subset \Jac(B)$.
\end{conj}

To tackle this conjecture, we analyze how to read the multiplication profile of a specific point $p-q$ from the affine equation $y^2=f(x)$ of the curve $B$. Without loss of generality, we assume $f(0)\neq 0$ and $p=(0,\sqrt{f(0)})$, where $\sqrt{f(0)}$ is a fixed square root of $f(0)$. 

\begin{defn}
Given an effective divisor $D$ on $B$ of degree $d$, let $p(x)-q(x)y$ be a rational function whose associated principal divisor is $D-dq$. We shall refer to $d$ as the \textit{support length} of this function.
\end{defn}

The $j$-th entry $\ep_j$ of $m([p-q])$ is codified by a rational function $p_j(x)-q_j(x)y\in K(B)$ that has minimal support length $d$ among all functions that vanish in $p$ to order are at least $j$ and are regular away from $q$. For this minimizing value of $d$, we have $\ep_j=d-j$; moreover, the pair $(p_j,q_j)$ is well-defined up to a scalar multiple. Hereafter, we let $\sum_{i=0}^{\infty}a_ix^i$ denote the image of $y$ in $\widehat{\mO}_{\pp^1,0}=k[\![x]\!]$ under the natural map $\mO_{B,p}\to \mO_{\pp^1,0}\to \widehat{\mO}_{\pp^1,0}$ induced by the hyperelliptic structure map $B\to\pp^1$. 

\begin{lem}\label{lem:p_j,q_j}
Up to a scalar multiple, $(p_j,q_j)$ is the unique pair of polynomials $(p(x),q(x))$ for which $x^j|(p-(\sum a_ix^i)\cd q)$ that minimizes the quantity
\[
\max\{2\deg(p), 2\ga+1+2\deg(q)\}.
\]
\end{lem}

\begin{proof}
The map $\mO_{B,p}\to \mO_{\pp^1,0}$ induced by the hyperelliptic structure map sends $m_p$ to $m_0=(x)$; so
the rational function $p(x)-q(x)y \in K(B)$ vanishes in $p$ to order at least $j$ if and only if $x^j$ divides $p(x)-(\sum a_ix^i)\cd q(x)$.
On the other hand, $p(x)-q(x)y$ plainly has support length $\max\{2\deg(p), 2\ga+1+2\deg(q)\}$.

\end{proof}
\begin{ex}The following are immediate consequences of Lemma~\ref{lem:p_j,q_j}:
\begin{enumerate}
    \item For every $j=1\hh\ga$, we have $(p_j(x),q_j(x))=(x^j,0)$.
    \item $(p_{\ga+1}(x),q_{\ga+1}(x))=(\sum_{i=0}^{\ga}a_ix^i,1)$.
    \item $(p_{\ga+2}(x),q_{\ga+2}(x))=(\sum_{i=0}^{\ga+1}a_ix^i,1)$.
\item Suppose $[p-q]$ is $N$-torsion. Then $\deg p_N(x)=\frac{N}{2}$ and $\deg q_N(x)\le\frac{N-(2\ga+2)}{2}$ whenever $N$ is even; while 
$\deg p_N(x)\le \frac{N-1}{2}$ and $\deg q_N(x)=\frac{N-(2\ga+1)}{2}$ whenever $N$ is odd.
\end{enumerate}
\end{ex}

\begin{defn}
We say that $\ep_{j_0}$ is a \textit{critical entry} of a multiplication profile $(\ep_j)$ whenever $\ep_{j_0}<\min(\ep_{j_0-1},\ep_{j_0+1})$. 
\end{defn}

Whenever $\ep_j$ is a critical entry, $\dv(p_j(x)-q_j(x)y)$ is not supported at $\iota(p)$. This means that $x$ divides neither $p_j$ nor $q_j$. Furthermore, the fact that $\dv(p_j(x)-q_j(x)y)$ has minimal support length means that $p_j$ and $q_j$ cannot share other linear terms; so  
$\gcd(p_j,q_j)=1$. Now set $i=j+\ep_j-\ga$. We then have 
$p_i=p_{i+1}=\dots=p_j$ and $q_i=q_{i+1}=\dots=q_j$, while 
$p_{j+k}=x^k p_j$ and $q_{j+k}=x^k q_j$ for $k=1\hh \ga-\ep_j$.

\begin{prop}\label{prop:coeff_prof}
Fix $n\ge\ga+2$. When $n+\ga$ is odd (resp. even), we have $\ep_n<\ga$ if and only if $D_n=0$, where
\[
D_n:=
\begin{vmatrix}
a_{\ga+2}&a_{\ga+3}&\hdots& a_{\frac{n+\ga+1}{2}}\\
a_{\ga+3}&a_{\ga+4}&\hdots& a_{\frac{n+\ga+3}{2}}\\
\vdots&\vdots&& \vdots\\
a_{\frac{n+\ga+1}{2}}&a_{\frac{n+\ga+3}{2}}&\hdots& a_{n-1}
\end{vmatrix} \qquad (\text{resp. }  D_n:=\begin{vmatrix}
a_{\ga+1}&a_{\ga+2}&\hdots& a_{\frac{n+\ga}{2}}\\
a_{\ga+2}&a_{\ga+3}&\hdots& a_{\frac{n+\ga+2}{2}}\\
\vdots&\vdots&& \vdots\\
a_{\frac{n+\ga}{2}}&a_{\frac{n+\ga+2}{2}}&\hdots& a_{n-1}
\end{vmatrix}).
\]
\end{prop}
\begin{proof}
The idea is to calculate $p_n$ and $q_n$ explicitly. 
Since the argument in each of the two cases is nearly identical, we only discuss the case in which $n+\ga$ is odd. To begin, recall that the support length of $p_n(x)-q_n(x)y$ is $\max\{2\deg p_n,2\ga+1+2\deg q_n\}$. For this quantity to be smaller than $\ga+n$, we must have $\deg q_n<\frac{n-\ga-1}{2}$. Now set $q_n(x)=\sum c_ix^i$. Then $(\sum a_ix^i)q_n(x)=\sum e_kx^k$, where $e_k=\sum_{i+j=k}c_ia_j$. Note that $p_n(x)$ is merely an appropriate truncation of $(\sum a_ix^i)q_n(x)$. Clearly we also require $\deg p_n<\frac{n+\ga}{2}$; it follows that $e_{\frac{n+\ga+1}{2}}=\hdots=e_{n-1}=0$. By definition of the $e_k$, the linear system of equations
\[\begin{bmatrix}
a_{\ga+2}&a_{\ga+3}&\hdots& a_{\frac{n+\ga+1}{2}}\\
a_{\ga+3}&a_{\ga+4}&\hdots& a_{\frac{n+\ga+3}{2}}\\
\vdots&\vdots&& \vdots\\
a_{\frac{n+\ga+1}{2}}&a_{\frac{n+\ga+3}{2}}&\hdots& a_{n-1}
\end{bmatrix}\begin{bmatrix}
x_{\frac{n-\ga-3}{2}}\\
x_{\frac{n-\ga-1}{2}}\\
\vdots\\
x_0
\end{bmatrix}=0\]
has a non-zero solution, and the result follows.
\end{proof}
\begin{rem}\label{rem:first_entries}
An upshot of Proposition~\ref{prop:coeff_prof} is that $\ep_{\ga+2}=\ga-1$ if and only if $a_{\ga+1}=0$; and $\ep_{\ga+3}<\ga$ if and only if $a_{\ga+2}=0$.
\end{rem}

From our previous discussion about restrictions on multiplication profiles, it is clear that whenever there are indices $j>i$ for which $\ep_i=\ep_j=\ga$ and $\ep_k<\ga$ for all $k$ strictly between $i$ and $j$, we in fact have
\[(\ep_i,\ep_{i+1}\hh \ep_{j-1},\ep_j)=(\ga,\ga-1\hh\ga-\frac{j-i}{2}+1, \ga-\frac{j-i}{2}, \ga-\frac{j-i}{2}+1\hh \ga-1,\ga).\]
Consequently, 
every multiplication profile is uniquely characterized by the vanishing or non-vanishing of the corresponding determinants $D_n$.

\subsection{Multiplication profiles of torsion points}\label{sec:mult_profiles_of_torsions}
Suppose $x=[p-q]\in \Theta^0_1$ is an $N$-torsion point. Its multiplication profile $m(x)=(\ep_j)$ is then determined by $(\ep_{\ga+2}\hh \ep_{\lfloor\frac{N}{2}\rfloor})$. Moreover, the following rules apply:
\begin{enumerate}
\item $\ep_j=0$ if and only if $N|j$;
    \item $\ep_{\ga+2} \in \{\ga-1,\ga\}$;
    \item $\ep_i+\ep_j\ge \ep_{i+j}$;
    \item there is no $j$ for which $(\ep_j,\ep_{j+1},\ep_{j+2})=(a,a+1,a)$ for any $a$; and
    \item there is no $j$ for which $(\ep_j,\ep_{j+1})=(a,a)$ with $a<\ga$. 
\end{enumerate}
In particular, we have either $\ep_{{N\over 2}}=\ep_{{N\over 2}-1}=\ga$ or $\ep_{{N\over 2}}=\ep_{{N\over 2}-1}-1$ when $N$ is even; and $\ep_{{N-1\over 2}}=\ga$ when $N$ is odd.

\medskip
Now suppose that $N\ge 2\ga+1$, and that $y^2=f(x)$ is the affine equation of a hyperelliptic curve with an $N$-torsion divisor $p-q$, where $p=(0,\sqrt{f(0)})$ and $\sum a_ix^i$ is a chosen square root of $f(x)$ inside $k[\![x]\!]$; then as explained previously, there exist polynomials $p_N(x),q_N(x)$ 
for which $x^N$ divides $p_N(x)-(\sum a_ix^i)q_N(x)$. 
Conversely, if our aim is to construct $f(x)$ so that the associated hyperelliptic curve has a divisor of particular torsion order, one possible approach involves specifying an approximating power series $\sum a_ix^i$ with certain properties. This approach serves our purposes well, as we may 
control $m([p-q])$ by imposing conditions on the coefficients $a_i$. 

\begin{prop}\label{prop:multiplication_cond}
Given $p$, $q$, $f(x)$, and $\sum a_ix^i$ as before, suppose that $[p-q]$ has order precisely $N$. Let $(\ep_j)_{j=1}^{\infty}$ be a sequence satisfying the five necessary conditions stated in the beginning of this subsection. Then $m([p-q])=(\ep_j)$ if and only if $\ep_j=0$ for precisely those $j$ between $\ga+1$ and $\lfloor\frac{N}{2}\rfloor$ for which $\ep_j<\ga$.
\end{prop}

\begin{proof}
From our earlier discussion, it is clear that the sequence $(\ep_{\ga+1}\hh \ep_{\lfloor\frac{N}{2}\rfloor})$ determines the full multiplication profile in this particular case. The desired conclusion now follows from Proposition~\ref{prop:coeff_prof} and the discussion following it.
\end{proof}
We are now ready to give a complete list of conditions that characterize those power series $\sum a_i x^i$ that give rise to a two-pointed hyperelliptic curve $(B,p,q)$ for which the order of $[p-q]$ is precisely $N$ and $m([p-q])=(\ep_j)_{j=1}^{\infty}$. A key observation is that for every fixed value of $N$, finding a power series that satisfies these conditions is a {\it finitely-determined} problem. Namely, every coefficient $a_j$ with $j \geq N$ is uniquely determined by $a_0\hh a_{N-1}$; so the set of admissible power series is a sublocus of the projective space $\pp^{N-1}$ in these coordinates.

\medskip
\noindent \textbf{Conditions for admissibility:}

\begin{itemize}
\item (Polynomiality and degree): For every $k\ge 2\ga+2$, we have $\sum_{i+j=k} a_ia_j=0$, while $\sum_{i+j=2\ga+1} a_ia_j\neq0$.

\item (Separability):
The Sylvester determinant $S(f)(a_0\hh a_{N-1})$ is nonzero (see Notation~\ref{nt:sylv}), where $f(x)=\sum_{k=0}^{2\ga+1}(\sum_{i+j=k}a_ia_j)x^k$.

\item (Torsionness) Whenever $N$ is even (resp. odd), the Hankel matrix
\begin{equation}\label{eq:hankel_non-maximality}
M_N:=\begin{bmatrix}
a_{\ga+2}&a_{\ga+3}&\hdots& a_{\frac{N}{2}+1}\\
a_{\ga+3}&a_{\ga+4}&\hdots& a_{\frac{N}{2}+2}\\
\vdots&\vdots&& \vdots\\
a_{\frac{N}{2}+\ga}&a_{\frac{N}{2}+\ga+1}&\hdots& a_{N-1}
\end{bmatrix}\text{ (resp., }
M_N:=\begin{bmatrix}
a_{\ga+1}&a_{\ga+2}&\hdots& a_{\frac{N+1}{2}}\\
a_{\ga+2}&a_{\ga+3}&\hdots& a_{\frac{N+3}{2}}\\
\vdots&\vdots&& \vdots\\
a_{\frac{N-1}{2}+\ga}&a_{\frac{N+1}{2}+\ga}&\hdots& a_{N-1}
\end{bmatrix}\text{)}
\end{equation}
has less-than-maximal rank. 
\item (Profile-specific conditions): See Proposition \ref{prop:multiplication_cond}.
\end{itemize}

\begin{defn}
Given $N\ge 2\ga+3$, let $\ti{T}_N$ denote the closed subset of $\pp^{N-1}$ given by the intersection of the closed subset defined by the homogeneous equations $\sum_{i+j=k}x_ix_j=0$ for $k=2\ga+2\hh N-1$ with the closed subset given by the vanishing of maximal minors of the $(\frac{N}{2}-1)\times(\frac{N}{2}-\ga)$ Hankel matrix $M_N$. 
Let $T_N$ denote the locally-closed subset given by the intersection of $\ti{T}_N$ with the open loci $(\sum_{i+j=2\ga+1}x_ix_j\neq 0)$ and 
$(S(f)(x_0\hh x_{N-1})\neq 0)$ (see Notation \ref{nt:sylv}), where $f(x)=\sum_{k=0}^{2\ga+1}(\sum_{i+j=k}x_ix_j)x^k$.
\end{defn}

We think of $T_N$ as a parameter space for square roots of separable polynomials $f(x)$ of degree $2\ga+1$ whose corresponding hyperelliptic curves admit $N$-torsion points $p=(0,\sqrt{f(0)})$. Note that the definition of $f$ guarantees that $a_0\neq 0$. 



\subsubsection{Conjectural classification of multiplication profiles for torsion orders $N \leq 2\ga+6$}


According to \cite[Proposition 4.3]{eisenbud1988linear}, whenever $N$ is even (resp. odd) the maximal minors of $M_N$ cut out a closed subscheme of $\PP^{N-1}$ that is a cone over 
an integral, normal closed subscheme $W$ of $\PP^{N-\ga-2}_{(x_{\ga+2}:\hdots:x_{N-1})}$ (resp. $\PP^{N-\ga-1}_{(x_{\ga+1}:\hdots:x_{N-1})}$) of codimension $\ga$. In particular, the ideal $I$ generated by the maximal minors of $M_N$ is a homogeneous prime ideal generated by some homogeneous elements of degree $\lfloor\frac{N+1}{2}\rfloor-\ga$. 

\medskip
In light of \cite[Cor. 2.8]{zarhin2019division}, we may assume that $N \geq 2\ga+1$ without loss of generality. Moreover, as explained in Examples~\ref{ex:(2ga+1)-torsion}, \ref{ex:torsion1}, and \ref{ex:(2ga+5)-torsion} below, whenever $N \in \{2\ga+1,2\ga+2,2\ga+3,2\ga+5\}$ a single multiplication profile is possible; so the interesting cases are those for which $N=2\ga+4$ or $N=2\ga+6$. On the other hand, Example~\ref{ex:torsion2} (resp., Example~\ref{ex:torsion4}) establishes that there are two (resp., three) possible multiplication profiles whenever $N=2\ga+4$ (resp., $N=2\ga+6$). Experimental computer evidence suggests that each of these profiles determines a distinguished component of the corresponding parameter space $T_N$.

\begin{conj}\label{conj:torsion-order_profiles}
Suppose that either $N=2\ga+4$ or $N=2\ga+6$. Each irreducible component $V_i$ of the parameter space $\ti{T}_N$ is cut out by a prime ideal $I_i$ generated by the \emph{polynomiality quadrics} $\sum_{i=0}^{N-1} x_i x_{k-i}$, $k=2\ga+2,\dots,N$ and maximal minors of the Hankel matrix $M_N$ in \eqref{eq:hankel_non-maximality}, together with a distinguished set of auxiliary generators. More precisely, we have
\[
\ti{T}_{2\ga+4}= V_1 \cup V_2
\]
where $I_1=I(V_1)$ is generated by the polynomiality quadrics, $2 \times 2$ Hankel minors, and the quadrics $2x_{\lfloor \fr{\ga+1+j}{2}\rfloor} x_{\lfloor \fr{\ga+1+j}{2}\rfloor+1} -x_{\ga+1}x_j$, $j=\ga+3,\dots, 2\ga+3$; while $I_2=I(V_2)$ is generated by the polynomiality quadrics, maximal Hankel minors, and $x_{\ga+2}$. Similarly, we have a decomposition
\[
\ti{T}_{2\ga+6}= V_1 \cup V_2 \cup V_3.
\]
This time, $I_1=I(V_1)$ is generated by the polynomiality quadrics, $3 \times 3$ Hankel minors, and the cubics 
\[
-2x_{\ga+3} D^{(1,2)}_{\ell_1,\ell_2}+ x_{\ga+2} D^{(1,3)}_{\ell_1,\ell_2}, (\ell_1,\ell_2) \in \binom{[\ga+2]}{2}
\]
in which $D^{(j_1,j_2)}_{\ell_1,\ell_2}$ denotes the $2 \times 2$ determinant associated to rows $(\ell_1,\ell_2)$ and columns $(j_1,j_2)$ of $M_{2\ga+6}$. The ideal $I_2=I(V_2)$ is generated by the polynomiality quadrics, $3 \times 3$ Hankel minors, the cubics
\[
2x_{\ga+3}D^{(1,2)}_{\ell_1,\ell_2}+ x_{\ga+1}D^{(2,3)}_{\ell_1,\ell_2}
\]
in which $(\ell_1,\ell_2) \in \binom{[\ga+2]}{2}$ indexes a pair of rows with $\ell_i \geq 2$, $i=1,2$, the cubics 
\[
2x_{\ga+3}(x_{\ga+1}x_{j+2}-x_{\ga+3}x_j)-x_{\ga+1}(x_{\ga+4} x_{j+1}+ x_{\ga+3} x_{j+2}), j=\ga+3,\dots,2\ga+3
\]
and $x_{\ga+2}$. The ideal $I_3=I(V_3)$ is generated by the polynomiality quadrics, $3 \times 3$ Hankel minors, and $x_j$, $j=\ga+1,\ga+2$.
\end{conj} 

\begin{rem}
Note that Conjecture~\ref{conj:torsion-order_profiles} is coherent with the numerology of Examples~\ref{ex:torsion2} and \ref{ex:torsion4}, respectively. Indeed, the order of the components we have chosen is in increasing order of speciality here, as it is there. The conjecture predicts that a general point of $V_i$ corresponds to an $N$-torsion point whose multiplication profile is (uniquely determined by) the $i$-th sequence listed in the corresponding example.
\end{rem}

\section{Explicit realizations and codimension estimates}\label{sec:explicit_realizability}
\begin{defn}\label{defn:potential}
    Fix positive integers $N$ and $\ga$. Let $(\ep_j)_{j=1}^{N-1}$ be any sequence of non-negative integers. We say it is a \textbf{potential multiplication profile} of an $N$-torsion point on a hyperelliptic curve of genus $\ga$ if it is the subsequence of an infinite sequence $(\ep_j)_{j=1}^{\infty}$ satisfying all conditions in Proposition \ref{prop:profile_conditions} such that $\ep_N=0$. Whenever there exists 
    a hyperelliptic curve of genus $\ga$ with an $N$-torsion point whose multiplication profile is $(\ep_j)_{j=1}^{\infty}$, we say the potential multiplication profile is \textbf{realizable}.
\end{defn}
In this section, we classify all potential multiplication profiles up to $N=2\ga+9$ for arbitrary $\ga$ and show that all but six of them (in orders $2\ga+3$, $2\ga+5$, $2\ga+8$, and $2\ga+9$ resp.) are actually realizable as multiplication profiles of a torsion point.  
\begin{ex}\label{ex:(2ga+1)-torsion}
Let $x=[p-q]$ be a $2$-torsion (resp. $(2\ga+1)$-torsion,  $(2\ga+3)$-torsion) point. Then $m(x)$ is determined by $m(x)_2=(1,0)$ (resp. $m(x)_{2\ga+1}=(1,2\hh\ga,\ga\hh 1,0)$, $m(x)_{2\ga+3}=(1,2\hh\ga,\ga,\ga,\ga\hh 1,0)$). Since order-$2$ points are just Weierstrass points, and order-$(2\ga+1)$ points are known to exist on certain hyperelliptic curves of genus $\ga$ \cite{bekker2020torsion}, the first two potential multiplication profiles are realizable. 
\end{ex}

\begin{ex}\label{ex:torsion1}
Here we compute the multiplication profiles of those torsion points constructed in Example \ref{ex:torsion_1}. To do so, we start by reducing the Mumford presentation $(x^{\ga+s}, v(x)=\sqrt{\frac{-1}{2}}(-x^{\ga-s+1}+1))$ using Cantor's algorithm;
we obtain $\frac{f(x)-v^2(x)}{x^{\ga+s}}=c(x^{\ga-s+1}-1)$.
It follows that $(\ga+s)\cd (p-q)\sim \sum ((\xi_j,0)-q)$, where $\xi_j$ runs through all $(\ga-s+1)$-th roots of unity. In particular, $\ep_{\ga+s}=\ga-s+1$ and the multiplication profile of $p-q$ is determined by 
\[m(p-q)_{2\ga+2s}=(1\hh\ga,\ga,\ga-1\hh \ga-s+1\hh\ga,\ga\hh 1,0).\] 
Note that any $(2\ga+2)$-torsion point must have the latter profile, with $s=1$. Thus, the potential multiplication profile of a $(2\ga+2)$-torsion point is realizable. 
\end{ex}
\noindent The following observation is obvious, yet useful in practice:
\begin{lem}
For every even number $N\ge 2\ga+1$, $[p-q]$ has order $N$ if and only if $\frac{N}{2}(p-q)$ has order 2. 
\end{lem}
On the other hand, elements of $\Jac(B)[2]$ have $q$-reduced presentations of the form $\sum_{i=1}^s q_i-sq$, where $q_i$ are pairwise-distinct Weierstrass points distinct from $q$. Finally, up to linear change of coordinates, we may always assume that the $x$-coordinate of any fixed torsion point is 0.  Together these three observations often enable us to simplify the equations involved in determining the multiplication profiles of torsion points of even order. 

\begin{ex}\label{ex:torsion2}
Fix $s,r\ge 0$ such that $s+r=\ga+1$. For a general choice of $(a_1\hh a_s,b_1\hh b_r)$,  Lemma~\ref{lem:poly_sep} implies that $x^{2\ga+2-s}-\prod_{i=1}^{s}(x-a_i)\cd\prod_{j=1}^r (x-b_j)^2$ is separable of degree $2\ga+1-s$.\footnote{The genericity assumption guarantees that $e_1(a_\hh a_s,b_1,b_1\hh b_r,b_r)\neq 0$, and the conclusion on the degree of the difference follows.} Letting $c=(\sum a_i+2\sum b_j)^{-1}$, it follows that
\[f(x)=c\cd\prod_{i=1}^{s}(x-a_i)[x^{2\ga+2-s}-\prod_{i=1}^{s}(x-a_i)\cd\prod_{j=1}^r(x-b_j)^2]\] 
is a monic, separable polynomial of degree $2\ga+1$. 
Now let 
\[
v(x)=\sqrt{-c}\prod_{i=1}^{s}(x-a_i)\cd\prod_{j=1}^r(x-b_j), \text{ so that } 
f(x)-v^2(x)=c\cd\prod_{i=1}^{s}(x-a_i)\cd x^{2\ga+2-s}. 
\]
Letting $p=(0,v(0))$, we see that $(2\ga+2-s)(p-q)\sim \sum_{i=1}^{s}((a_i,0)-q)$ is a 2-torsion point of the Jacobian of $B: y^2=f(x)$.

\medskip
When $s=\ga$, it follows that $x=p-q$ is of order $2\ga+4$ and its multiplication profile is determined by $m_{2\ga+4}(x)=(1,2\hh\ga,\ga,\ga,\ga,\ga,\ga-1\hh1,0)$. This profile and that of Example~\ref{ex:torsion1} (with $s=2$) are the only two possible profiles for a $(2\ga+4)$-torsion element in $\Theta^0_1$ on some hyperelliptic Jacobian. 
In light of Example \ref{ex:torsion_1}, we conclude that both potential multiplication profiles of a $(2\ga+4)$-torsion point are realizable. 

\medskip
More generally, for every $s=1\hh \ga$, $x=p-q$ is of order $4\ga+4-2s$ and its multiplication profile is determined by \[m_{4\ga+4-2s}(x)=(1,2\hh\ga,\ga,\ga,\ga-1\hh s,s+1\hh \ga,\ga,\ga\hh1,0).\]

\end{ex}

\begin{ex}\label{ex:(2ga+5)-torsion}
Say that $x=p-q$ is a $(2\ga+5)$-torsion point; then $m(x)$ is determined by 
\[m(x)_{2\ga+5}=(1\hh\ga,\ga,\ga,\ga,\ga,\ga \hh 0).\]
Indeed, according to Example \ref{ex1}, every $\ep_j$ is already determined, except for $\ep_{\ga+2}=\ep_{\ga+3}$. On the other hand, notice that $2\cd(\ga+2)\cd x=[\iota(p)-q]\in\Theta^0_1$; it follows from \cite[Theorem 2.5] {zarhin2019division} that $\ep_{\ga+2}=\ep_{\ga+3}=\ga$. (Alternatively, use condition (4) in Conjecture~\ref{conj:profile}.)
\end{ex}

\begin{ex}\label{ex:torsion3}
Fix $1\le s\le \ga-1$, and let $N=2\ga+1-s$. By Lemma \ref{lem:sep2}, $g(x)=cx^N+\prod_{i=1}^s(x-a_i)\cd(x-b)^2$ is separable of degree $N$, for general choices of $a_1\hh a_s,b,c$. Let $a_{s+1}\hh a_{2\ga+1}$ denote the roots of $g$; then $a_1\hh a_{2\ga+1}$ are all distinct. Let 
\[
f(x)=\prod_{i=1}^{2\ga+1}(x-a_i) \text{ and }v(x)=\sqrt{c^{-1}}\prod_{i=1}^{s}(x-a_i)(x-b).
\]

\noindent Set $p=(0,v(0))$. Reducing $(x^N,v(x))$ one time using Cantor's algorithm, we conclude that $N(p-q)\sim \sum_{i=1}^s((a_i,0)-q)$. Consequently, $x=p-q$ has order $4\ga+2-2s$ and $m(x)$ is determined by 
\[m(x)_{4\ga+2-2s}=(1\hh\ga,\ga,\ga-1\hh s,s+1\hh \ga,\ga,\ga-1\hh 0).\]
This is the multiplication profile of a $(4\ga+2-2s)$-torsion point in $\Theta_1$ 
for which $\ep_{2\ga+1-s}$ is minimized.
\end{ex}

\begin{ex}\label{ex:torsion4}
Let $N=2\ga+6$. By symmetry, potential multiplication profiles associated with torsion points of this order are determined by either
\begin{enumerate}
    \item $(\ep_{\ga+2},\ep_{\ga+3})=(\ga,\ga)$;
    \item $(\ep_{\ga+2},\ep_{\ga+3})=(\ga,\ga-1)$; or
    \item $(\ep_{\ga+2},\ep_{\ga+3})=(\ga-1,\ga-2)$.
\end{enumerate}

The second (resp., third) possibility is realized in Ex.~\ref{ex:torsion2} (resp., Ex.~\ref{ex:torsion3}) with $s=\ga-1$ (resp., $s=\ga-2$).
We claim the first possibility is also realizable. Indeed, set $s=\ga+1$ and $\ell=3$; and apply Lemma~\ref{lem:sep3} to obtain a separable polynomial of degree $\ga$ of the form $g(x)=x^{\ga+3}-\prod_{i=1}^s(x-a_i)\cd(x-b)^2:=c\cd \prod_{i=\ga+2}^{2\ga+1}(x-a_i)$. The generality of the $a_i$ together with the fundamental theorem of symmetric polynomials
guarantees that the constants $a_1\hh a_{\ga+1},b$ are all non-zero; and it's easy to see that $a_1\hh a_{2\ga+1}$ are pairwise distinct, provided that $a_1\hh a_{\ga+1}$ are. In the case at hand, to ensure that the $a_i$, $i=1,\dots,\ga+1$ are pairwise distinct, we use the fact that $4e_2(x_1\hh x_{\ga+3})-3e_1^2(x_1\hh x_{\ga+3})=0$ admits a solution $(a_1\hh a_{\ga+3})$ whose entries are pairwise distinct. Setting $b=-\frac{1}{2}(\sum a_i)$, we obtain $e_1(a_1\hh a_{\ga+3},b,b)=e_2(a_1\hh a_{\ga+3},b,b)=0$.  

\medskip
Now let $f(x)=\prod_{i=1}^{2\ga+1}(x-a_i)$, $u_0(x)=x^{\ga+3}$, $v_0(x)=\sqrt{-c}\prod_{i=1}^{\ga+1}(x-a_i)\cd(x-b)$. 
Apply the reduction algorithm once to the pair $(u_0,v_0)$, we obtain $u_1(x)=\prod_{i=1}^{\ga+1}(x-a_i)$. In particular, we have $(\ga+3)\cd(p-q)\sim \sum_{i=1}^{\ga+1}(q_i-q)$, where $p=(0,v(0))$ and $q_i=(a_i,0)$. We conclude that $x=p-q$ has order $2\ga+6$ and 
$\dv(y)=\sum_{i=1}^{2\ga+1}(q_i-q)$; so $m(x)$ matches the first candidate profile on our list.
\end{ex}

\begin{ex}\label{ex:odd1}
Let $N=2\ga+7$. By symmetry, potential multiplication profiles are indexed by
\begin{enumerate}
    \item $(\ep_{\ga+2},\ep_{\ga+3})=(\ga,\ga)$; \text{ and }
    \item $(\ep_{\ga+2},\ep_{\ga+3})=(\ga-1,\ga)$.
\end{enumerate}
Both profiles are realized in Appendix \ref{sec:flynn}. For the first (resp., second), we require that $\ga\ge 3$ (resp., $\ga\ge 4$).   
\end{ex}

\begin{ex}\label{ex:2ga+8}
When $N=2\ga+8$, potential profiles are indexed by
\begin{enumerate}
    \item $(\ep_{\ga+2},\ep_{\ga+3},\ep_{\ga+4})=(\ga,\ga,\ga)$;
    \item $(\ep_{\ga+2},\ep_{\ga+3},\ep_{\ga+4})=(\ga,\ga-1,\ga-2)$;
    \item $(\ep_{\ga+2},\ep_{\ga+3},\ep_{\ga+4})=(\ga,\ga,\ga-1)$;
    \item $(\ep_{\ga+2},\ep_{\ga+3},\ep_{\ga+4})=(\ga-1,\ga-2,\ga-3)$; \text{ and}
    \item $(\ep_{\ga+2},\ep_{\ga+3},\ep_{\ga+4})=(\ga-1,\ga,\ga)$.
\end{enumerate}

Of these, the second is realized in Ex.~\ref{ex:torsion2}, setting $s=\ga-2$. The third may be realized using a similar construction applying Lemma \ref{lem:sep3} as used to realize the first listed profile in Ex.~\ref{ex:torsion4}. 
The fourth is realized in Ex.~\ref{ex:torsion3}, setting $s=\ga-3$.

\medskip
According to Proposition~\ref{prop:coeff_prof}, to realize the first profile we need to find a power series $\sum a_ix^i$ in $R_{2\ga+8}$ for which $a_{\ga+1}\neq 0$, $a_{\ga+2}\neq 0$, and $a_{\ga+1}a_{\ga+3}-a_{\ga+2}^2\neq 0$. 
Similarly, in order to realize the fifth profile, the required conditions are $a_{\ga+1}=0$ and $a_{\ga+2}\neq 0$. 
\end{ex}

\begin{ex}\label{ex:2ga+9}
When $N=2\ga+9$, candidate profiles are indexed by
\begin{enumerate}
    \item $(\ep_{\ga+2},\ep_{\ga+3},\ep_{\ga+4})=(\ga,\ga,\ga)$;
    \item $(\ep_{\ga+2},\ep_{\ga+3},\ep_{\ga+4})=(\ga,\ga-1,\ga)$; \text{ and}
    \item $(\ep_{\ga+2},\ep_{\ga+3},\ep_{\ga+4})=(\ga-1,\ga,\ga)$.
\end{enumerate}  
The first profile is realized in Appendix~\ref{sec:flynn} for every $\ga\ge 4$.
\end{ex}

For the convenience of the reader, we summarize the above discussion and its relation to the question of realizability of numerical semigroups in a single statement:
\begin{thm}\label{thm:torsion}
 The number $p(N)$ of potential multiplication profiles of an $N$-torsion point on a genus-$\ga$ hyperelliptic curve for $N\le 2\ga+9$ are as follows:

 \medskip
 \begin{tabular}{c|c|c|c|c|c|c|c|c|c|c}
    $N$  & 2 &$2\ga+1$ &$2\ga+2$ &$2\ga+3$ &$2\ga+4$&$2\ga+5$ &$2\ga+6$&$2\ga+7$ &$2\ga+8$&$2\ga+9$ \\
    \hline
    $p(N)$ & 1 & 1&1&1&2&1&3&2&5&3 
 \end{tabular}

\medskip
Whenever $N \in\{2,2\ga+1,2\ga+2,2\ga+4,2\ga+6\}$ or $N=2\ga+7$ with $\ga\ge 4$, all potential multiplication profiles are realizable. Whenever $N=2\ga+8$, three out of five potential multiplication profiles are realizable. When $N=2\ga+9$ and $\ga\ge 4$, one out of three potential multiplication profiles are realizable.

\medskip
Whenever the $\ep$-vector (see Definition \ref{defn:ep_vec}) of an $(N,\ga)$-hyperelliptic semigroup coincides with one of the realizable multiplication profiles, that semigroup is Weierstrass-realizable by a degree-$N$ cyclic cover of a genus-$\ga$ hyperelliptic curve.  
\end{thm}

\begin{ex}\label{ex:fibonacci_numerology}
The classification of indexing sequences for potential multiplication profiles associated with torsion points of arbitrary even order $2\ga+2k$ is similar to that of Example~\ref{ex:2ga+8}; an upshot is that the number $N_k$ of candidate profiles associated with torsion order $2\ga+2k$ is the $(k+1)$-th Fibonacci number for every $k \geq 1$. Indeed, we claim that
\begin{equation}\label{eq:fibonacci_recursion}
N_k= \sum_{i=1}^{k-2} N_{i}+ 2
\end{equation}
for every positive integer $k \geq 3$. To see this, note that indexing sequences \\$(\ep_{\ga+2},\ep_{\ga+3}, \dots, \ep_{\ga+k})$ are of four possible types: $(\ga, \ga, \dots)$, $(\ga, \ga-1, \dots)$, $(\ga-1,\ga,\dots)$, and $(\ga-1,\ga-2,\dots)$.
In the first case, the remaining $(k-3)$ entries comprise an arbitrary profile associated with a torsion order of index $2(\ga+k-2)$, as manifested by the instance of $N_{k-2}$ on the right-hand side of \eqref{eq:fibonacci_recursion}. Similarly, in the third case, $\ep_{\ga+4}=\ga$ is forced, and the remaining $(k-4)$ entries comprise an arbitrary profile associated with a torsion order of index $2(\ga+k-3)$; these account for the instance of $N_{k-3}$. In the second and fourth cases, we have the following dichotomy: either the sequence is strictly decreasing (here there are precisely two possibilities, which accounts for the instance of 2 in \eqref{eq:fibonacci_recursion}) or it eventually reverses course and starts to increase. In the latter situation, the sequence starts increasing after $j$ steps, where $2 \leq j \leq \lfloor \frac{k}{2}\rfloor-1$, as symmetry requires that this first ascent end in $\ga, \ga$; 
the corresponding contribution to $N_k$ is $N_{k-2j}$ (resp., $N_{k-1-2j}$) in the second (resp., fourth) case. By varying the choice of $j$ we obtain 
\[
N_k= N_{k-2}+ \sum_{j=2}^{\lfloor \frac{k}{2} \rfloor-1} N_{k-2j}+ N_{k-3}+  \sum_{j=2}^{\lfloor \frac{k}{2} \rfloor-1} N_{k-1-2j} +2
\]
which is equivalent to \eqref{eq:fibonacci_recursion}.

\medskip
On the other hand, the number $\wt{N}_k$ of potential multiplication profiles associated with torsion points of arbitrary {\it odd} order $2\ga+2k+1$, $k \geq 1$ is equal to the $k$-th Fibonacci number. Indeed, an analysis similar to that of the previous paragraph yields
$\wt{N}_k= \sum_{i=1}^{k-2} \wt{N}_{i}+1$ for every $k \geq 3$.

\medskip
It is instructive to compare those Weierstrass semigroups derived from multiplication profiles of $N$-torsion point classes 
against the Carvalho--Torres semigroup ${\rm U}_0$ of Example~\ref{ex:sanity}, whose associated multiplication profile is the zero vector. To ensure that $d=d_{g,\ga,N}=\frac{(2g-2)-N(2\ga-2)}{N(N-1)}$ is an integer, we postulate that $g=\binom{N}{2}t+N(\ga-1)+1$, with $t=d \in \mb{Z}$; for reasons of ordering, we also assume $t \geq 2\ga+3$, and that $N$ and $\ga$ are even and nonzero. In every case, the nonzero elements $e_j$, $j=1,\dots,2N-1$ of the standard basis modulo $2N$ of the associated $2N$-semigroup are explicitly determined by the equations
\begin{equation}\label{eq:staircase_equations}
(Nt-1)j+ N \ep_j= \min\{e_{N-j},e_{2N-j}\} \text{ and } e_{N-j}+ e_{2N-j}= 2j(Nt-1)+ (2\ga+1)N
\end{equation}
for every $j=1,\dots,N-1$. For simplicity, we focus on the {\it staircase} of type $(\ep_1,\dots,\ep_N)=(1,2,\dots,\ga-1,\ga,\dots,\ga,\ga-1,\dots,2,1,0)$.  
From the minimal presentations and standard bases of the Carvalho--Torres and staircase semigroups, we may extract their $t$-asymptotic {\it effective weights} ${\rm ewt}$ in the sense of \cite[Def. 1.1]{Pfl}. We do this explicitly for the Carvalho--Torres semigroup ${\rm U}_0$; we will carry out the (significantly more complicated, yet similar) effect weight calculation for ${\rm S}_N$ elsewhere.

\begin{prop}\label{prop:CT_presentation_and_ewt}
Assume that $t \geq 2\ga+3$ is an odd integer, and that $\gamma$ is a strictly positive integer. The associated Carvalho--Torres semigroup ${\rm U}_0= {\rm U}_0(t)$ determined by the equations \eqref{eq:staircase_equations} has minimal presentation $\langle 2,2\ga+1 \rangle N+\langle Nt-1 \rangle$ and effective weight
${\rm ewt}({\rm U}_0)=(3\binom{N}{2}-(N-1))t+ N \ga- 6(N-1)$.
\end{prop}

\begin{proof}
\medskip
From the equations \eqref{eq:staircase_equations}, it is clear that ${\rm U}_0$ is minimally generated by $2N$, $(2\ga+1)N$, and $Nt-1$, and that its standard basis elements $e_i, i=0,\dots,2N-1$ are given by $e_0=2N, e_N=(2\ga+1)N$ and $\{e_{N-j},e_{2N-j}\}=\{(Nt-1)j, (Nt-1)j+(2\ga+1)N\}$ for every $j=1,\dots,N-1$. It follows that the gap set $\mb{N} \setminus {\rm U}_0$ is the union of $\{N, 3N, \dots, (2\ga-1)N\}$ and either
\[
\begin{split}
&\{N-j,3N-j, \dots, (tj-2)N-j\} \sqcup \{2N-j,4N-j,\dots,(tj+2\ga-1)N-j\} \text{ or }\\
&\{2N-j,4N-j,\dots,(tj-2)N-j\} \sqcup \{N-j,3N-j,\dots,(tj+2\ga-1)N-j\}
\end{split}
\]
depending on whether $tj$ is odd or even, i.e., on whether $j$ is odd or even (as $t$ is odd by assumption). We compute the effective weight of ${\rm U}_0$ by summing over every $j$ the subcontributions of $2N$, $2\ga+1$, and $Nt-1$ individually, i.e. for each of these minimal generators we compute the corresponding number of strictly-larger gaps. For example, when $j$ is odd, the subcontribution of $Nt-1$ from the first gap subset $\{ N-j,3N-j,\dots,(tj-2)N-j\}$ is equal to the number of integers $(2k+1)N-j$ strictly greater than $Nt-1$, where $0 \leq k < \frac{tj-1}{2}$. This, in turn, is precisely 
\[
\#\bigg\{k: \frac{t-1}{2}<k<\frac{tj-1}{2}\bigg\}= \frac{tj-1}{2}-\frac{t-1}{2}-\delta_{j>1}
\]
where $\delta_{j>1}= \delta_{j>1}(j)$ is 0 when $j=1$ and 1 otherwise. A straightforward count now yields that the effective weight contribution from gaps congruent to elements of $\{e_0,e_N\}$ is $\ga-1$, while the contributions from gaps congruent to elements $\{e_{N-j},e_{2N-j}\}$ are
\[
\sum_{j \text{ odd}} (t(3j-1)+ \ga-5- \delta_{j>1})
\text{ and} \sum_{j \text{ even}} (t(3j-1)+ \ga-6)
\]
respectively, where each sum is over indices $j \in \{1,\dots,N-1\}$, subject to the given parity restriction. Summing these three contributions yields the desired result.

\end{proof}

On the other hand, the standard basis elements $e_{N-j}$ and $e_{2N-j}$ of the staircase semigroup ${\rm S}_N$ are determined by \eqref{eq:staircase_equations} and the parity of $t$, and are recorded in Table~\ref{table:1}. 

\begin{prop}\label{prop:staircase_minimal_generation}
Assume $\ga$ is a strictly positive integer, $t \geq 2\ga+3$, and $N \geq 3\ga$. The staircase semigroup ${\rm S}_N={\rm S}_N(t)$ described by \eqref{eq:staircase_equations} is minimally generated by
{\small
\begin{equation}\label{min_gens_staircase_semigroup}
\{2N, (2\ga+1)N, N(t+1)-1\} \sqcup \{N(t-1)j-j+ (2\ga+1)N\}_{j=1}^{\ga} \sqcup \{Ntj+ N\ga-j\}_{j=\ga+1}^{N-\ga} \sqcup \{N(t-1)j+N^2-j\}_{j=N-\ga+1}^{N-1}.
\end{equation}
}
\end{prop}

\begin{proof}
\medskip
The hypotheses on $t$ and $\ga$ ensure that the putative minimal generators of ${\rm S}_N$ listed in \eqref{min_gens_staircase_semigroup} are ordered from smallest to largest. It is clear, moreover, that every minimal generator of ${\rm S}_N$ is a standard basis element; and that the smallest three standard basis elements $e_0=2N$, $e_N=(2\ga+1)N$, and $e_{2N-1}=N(t+1)-1$ are minimal generators. Furthermore, no element $N(t-1)j-j+(2\ga+1)N$ with $1 \leq j \leq \ga$ is realizable as a linear combination of $e_0$, $e_N$, and $e_{2N-1}$. Indeed, it were, the fact that $e_i+ e_j \equiv e_{i+j} \text{ (mod }2N)$ for every $i,j \in \{1,\dots,N-1\}$ implies it could be realized as a {\it pairwise} sum involving only two of these; namely, either $e_0+e_1$, $e_0+e_2$, or $e_1+e_2$; but this is absurd. The remaining minimal generators are of the form $\min(e_{N-j},e_{2N-j})$ with $\ga+1 \leq j \leq N-1$. To see that each of these elements is a minimal generator, it suffices to show that none of these is realizable as a pairwise sum of strictly-smaller putative generators. Concretely, this amounts to showing that for every triple of standard basis elements $(e_i,e_j,e_k)$ with $i+j \equiv k \text{ (mod }2N)$, we have $e_i+e_j>e_k$. This is straightforward, if tedious; here the numerical hypothesis on $N$ precludes the existence of triples $(e_i,e_j,e_k)$ with $i+j \equiv k \text{ (mod }2N)$ in which the indices $i$, $j$, and $k$ belong to the first, second, and third columns of Table~\ref{table:1}, respectively. Finally, no standard basis element $\max(e_{N-j},e_{2N-j})$ with $\ga+1 \leq j \leq N-1$ is a minimal generator, as it is equal to either $e_{2N-1}+e_{N-j+1}$ or $e_{2N-1}+e_{2N-j+1}$.
\end{proof}

We close by showing that, modulo suitable numerical hypotheses, any $(N,\ga)$-hyperelliptic semigroup ${\rm S}$ of genus $g$ 
that arises via our cyclic cover construction has effective weight $\text{ewt}({\rm S}) \geq g$.

\begin{prop}\label{prop:non-secundiveness}
Let $N \geq 2$, $\ga>0$ and $g \geq (2N-1)\ga$ be positive integers for which $t:=\frac{(2g-2)-N(2\ga-2)}{N(N-1)}$ is a positive integer, and assume that ${\rm S}$ is an $(N,\ga)$-hyperelliptic semigroup of multiplicity $2N$ indexed by an element of the feasible set $F(N)$ of Definition~\ref{defn:feasible}. 
The effective weight $\text{ewt}({\rm S})$ of ${\rm S}$ satisfies $\text{ewt}({\rm S})\geq g$ whenever $t>\frac{2\ga+5}{N-1}$, i.e., whenever $g>\frac{N}{2}(4\ga+3)+1$. 
\end{prop}

\begin{proof}
In light of \cite[Def. 5.1 and Lem. 5.3]{Pfl}, it suffices to show that ${\rm S}$ is {\it non-secundive}, i.e., that the largest gap of ${\rm S}$ is strictly greater than the sum of its two smallest minimal generators. By construction, these two smallest minimal generators are $2N$ and $(2\ga+1)N$, while the components of the standard basis $(2N,e_1,\dots,e_{N-1})$ of ${\rm S}$ modulo $2N$ satisfy $\min(e_{N-i},e_{2N-i}) \geq iN t$ for every $i=1,\dots,N-1$. In particular, this means that $\min(e_1,e_{N+1}) \geq N(N-1)t$, and therefore the largest gap $\ell$ of ${\rm S}$ satisfies $\ell \geq N(N-1)t-2N$. Our hypothesis on $t$ relative to $(N,\ga)$ now ensures that ${\rm S}$ is non-secundive.
\end{proof}



\begin{table}
\small
\begin{tabular}{c|c|c|c}
\empty & $1 \leq j \leq \ga$ & $\ga+1 \leq j \leq N-\ga$ & $N-\ga+1 \leq j \leq N-1$\\
\hline
  $t$ odd   &      
$e_{2N-j}=N(t+1)j-j$ and &   $j$ odd: $e_{N-j}=Ntj+ N\ga-j$ & $e_{2N-j}=N(t-1)j+N^2-j$ and\\
  \empty & $e_{N-j}=N(t-1)j-j+ (2\ga+1)N$ & and $e_{2N-j}=Ntj+ (\ga+1)N-j$ & $e_{N-j}=Nj(t+1)+(2\ga+1-N)N-j$ \\
  \empty & \empty & $j$ even: $e_{2N-j}=Ntj+ N\ga-j$ & \empty \\
  \empty & \empty & and $e_{N-j}=Ntj+ (\ga+1)N-j$\\
  \hline
    $t$ even & $j$ odd: $e_{N-j}=N(t+1)j-j$ and & $e_{2N-j}=Ntj+ N\ga-j$ and & $j$ odd: $e_{N-j}=N(t-1)j+N^2-j$ and\\
    \empty & $e_{2N-j}=N(t-1)j-j+ (2\ga+1)N$ & $e_{N-j}=Ntj+ (\ga+1)N-j$ & $e_{2N-j}=Nj(t+1)+(2\ga+1-N)N-j$\\
    \empty & $j$ even: $e_{2N-j}=N(t+1)j-j$ and & \empty & $j$ even: $e_{2N-j}=N(t-1)j+N^2-j$ and\\
    \empty & $e_{N-j}=N(t-1)j-j+ (2\ga+1)N$ & \empty & $e_{N-j}=Nj(t+1)+(2\ga+1-N)N-j$\\
\hline
\end{tabular}
\vspace{.1in}
\caption{Standard basis of the staircase semigroup}
\label{table:1}
\end{table}

\end{ex}

\appendix
\section{Separability results for univariate polynomials}\label{sec:separability}
A classical formula of Sylvester establishes that a univariate polynomial $f(x)\in k[x]$ has a multiple root if and only if the {\it resultant}
of $(f,f^{\pr})$ is zero. 
Here we derive some separability criteria for univariate polynomials of a particular form based on an analysis of the corresponding resultants, in order to justify the constructions carried out in the body of the paper.  

\begin{nt}\label{nt:sylv}
Let $S(f)$ denote the Sylvester polynomial of $f(x)=\sum_{i=0}^n a_ix^i$; that is, the determinant of the following $(2n-1)\times (2n-1)$ matrix:
{\small
\[\begin{bmatrix}
a_n& 0&\hdots&0     &na_n &0&\hdots&0 \\
a_{n-1}&a_n&\hdots&0&(n-1)a_{n-1}& na_n & \hdots&0 \\
\vdots&\vdots&&\vdots          &\vdots&\vdots&&\vdots \\
a_1&a_2&\hdots&a_{n-1}     &a_1 & 2a_2&\hdots& na_n \\
a_0&a_1&\hdots&a_{n-2}       &0 & a_1 &\hdots&(n-1)a_{n-1}\\
0&a_0 &\hdots& a_{n-3} &0  &0& \hdots &(n-2)a_{n-2} \\
\vdots&\vdots&& &\vdots&\vdots&\hdots&\vdots\\ 
0&0&\hdots&a_0 &0&0&\hdots&a_{1}\\
\end{bmatrix}.\]
}
\end{nt}

\begin{lem}\label{lem:poly_sep}
Let $(a_1\hh a_s,b_1\hh b_r)$ be a general point of $\ba^{s+r}(k)$ and $N=s+2r$. Then $f(x)=x^N-\prod_{i=1}^s(x-a_i)\cd \prod_{j=1}^r(x-b_j)^2$ is a separable polynomial.  
\end{lem}

\begin{proof}
Let $e_i$ denote the $i$-th elementary symmetric polynomial in $N$ variables, and let \\
$\ov{e_i}=e_i(-a_1\hh-a_s,-b_1,-b_1\hh-b_r,-b_r)$. Then $f(x)=-\sum_{i=1}^{N}\ov{e_i}x^{N-i}$.  By definition, the Sylvester polynomial $S(f)$ of $f(x)$ is homogeneous of degree $2N-3$ in $(\ov{e_1}\hh \ov{e_N})$, since $\deg(F)=N-1$. 

\medskip
\noindent Now let $F:\ba^{N}\to\ba^{N}$ denote the morphism 
\[
(a_1\hh a_s,b_1\hh b_{2r})\mapsto (\ov{e_i})_{i=1}^{N}=(e_i(-a_1\hh-a_s,-b_1\hh-b_{2r}))_{i=1}^{N}
\]
that sends the roots of a degree-$N$ polynomial to the coefficients of its expansion. Pulling $S(f)$ back along $F$, we obtain a polynomial $S$ that is homogeneous in $a_1\hh a_s,b_1\hh b_{2r}$ of degree $N^2-2$.

\medskip
To conclude, it suffices to show that $S$ restricts to a non-zero function on the linear subspace defined by $b_{2i-1}-b_{2i}=0$, for $i=1\hh r$. For this purpose, it further suffices to show that $S$ restricts to a non-zero function on the linear subspace defined by $b_1=\hdots =b_r=1$. In fact, 
the polynomial $x^N-(x-1)^N$ is separable, with roots $\frac{\xi^i}{\xi^i-1}$ ($i=1\hh N-1$), where $\xi$ is a primitive $N$-th root of unity; but this means exactly that $S(1\hh 1)\neq 0$.  
\end{proof}

Using basic properties of symmetric polynomials, we obtain a strengthened form of Lemma~\ref{lem:poly_sep}: 

\begin{lem}\label{lem:sep3}
Suppose $N=s+2$ and that $(a_1\hh a_s,b,b)$ is a general point of the algebraic subset of $\ba^{N}(k)$ defined by $e_1(x_1\hh x_N)=\hdots=e_{\ell-1}(x_1\hh x_N)=0$ and $x_{s+2}-x_{s+1}=0$, where $1\le \ell<N$.\footnote{When $\ell=1$, we interpret these to be the hypotheses of Lemma~\ref{lem:poly_sep}.}. Then $f(x)=x^N-\prod_{i=1}^s(x-a_i)\cd (x-b)^2$ is separable.  
\end{lem}

\begin{proof}
Under the given assumptions, we have $f(x)=-\sum_{i=\ell}^{N}\ov{e_i}x^{N-i}$; the corresponding Sylvester polynomial $S(f)$ is\footnote{Starting from the definition, add $-(N-\ell-1)$ times the $i$th column to the $(N-\ell-2+i)$th column for $i=1\hh N-\ell-2$.}
{\small
\[-
\begin{vmatrix}
\ov{e_{\ell}}& 0&\hdots&0     &0 &0&\hdots&0 \\
\ov{e_{\ell+1}}&\ov{e_{\ell}}&\hdots&0&-\ov{e_{\ell+1}}& 0 & \hdots&0 \\
\vdots&\vdots&&\vdots          &\vdots&\vdots&&\vdots \\
\ov{e_{N-2}}&\ov{e_{N-3}}&\hdots&\ov{e_{\ell}}     &-(N-\ell-2)\ov{e_{N-2}} & -(N-\ell-3)\cd\ov{e_{N-3}} &\hdots& 0 \\
\ov{e_{N-1}}&\ov{e_{N-2}}&\hdots&\ov{e_{\ell+1}}      &-(N-\ell-1)\cd\ov{e_{N-1}} & -(N-\ell-2)\ov{e_{N-2}} &\hdots&(N-\ell)\ov{e_{\ell}}\\
\ov{e_N}&\ov{e_{N-1}} &\hdots& \ov{e_{\ell+2}} &-(N-\ell)\ov{e_{N}}  &-(N-\ell-1)\cd\ov{e_{N-1}} & \hdots &(N-\ell-1)\cd\ov{e_{\ell+1}} \\
\vdots&\vdots&& &\vdots&\vdots&\hdots&\vdots\\ 
0&0&\hdots&\ov{e_{N}}&0&0&\hdots&\ov{e_{N-1}}\\
\end{vmatrix}.
\]
}
In particular, it is homogeneous of degree $2N-2\ell-1$ in $\ov{e_1}\hh \ov{e_N}$, and symmetric and homogeneous of degree $N^2-N-\ell^2$ in the affine coordinates $x_1\hh x_{N}$ of $\ba^N$.  

\medskip
Our aim is to show that the pull-back of $S(f)$ to the algebraic set $Z$ defined by $e_i(x_1\hh x_N)=0$ ($i=1\hh \ell-1$) and $x_{s+2}-x_{s+1}=0$ is not identically zero. 
Indeed, if $S(f)$ vanishes along $Z$, then because $S(f)$ is symmetric it must vanish on the algebraic set defined by $e_i(x_1\hh x_N)=0$, $i=1\hh \ell-1$ and $V^2_N=0$, where $V_N(x_1\hh x_N)=\prod_{i<j}(x_i-x_j)$ is the Vandermonde polynomial in $N$ variables. Thus, letting $I=\lag e_1\hh e_{\ell-1},V^2_N\rag $, we have $S(f)\in\sqrt{I}$, i.e., $S(f)^m=P_N(0\hh 0,e_{\ell}\hh e_N)\cd h(e_{\ell}\hh e_N)$ for some $m\ge 1$, where $V^2_N=P_N(e_1\hh e_N)$. 
However, $S(f)$ is irreducible in the variables $e_1\hh e_N$; see \cite[Example 1.1.4]{gelfand2008discriminants}. So the vanishing of $S(f)$ along $Z$ would imply that $P_N(0\hh 0,e_{\ell}\hh e_N)=S(f)^{m^{\pr}}$ for some $m^{\pr}$. This last possibility, however, is impossible as $P_N(0\hh 0,e_{\ell}\hh e_N)$ is {\it inhomogeneous} in $e_{\ell}\hh e_N$. Indeed, via Newton's method we may express $V_n^2(x_1\hh x_n)$ as a polynomial in elementary symmetric functions $\sum_{\la\vdash n(n-1)}c^{(n)}_{\la}e_{\la}$, where $e_{\la}=\prod e_{\la_i}$ and $c^{(n)}_{n^{n-1}}\neq 0$.\footnote{This is a consequence of the fundamental theorem of symmetric polynomials.} The fact that the coefficients $c^{(n)}_{n^{n-1}}$ are nonzero follows via induction using the relation $V_{n}^2=P_{n-1}(e_1\hh e_{n-1})e^2_{n-1}+e_{n}h_{n}$.
\end{proof}

The next lemma builds on the same idea, and its proof is even simpler. 
\begin{lem}\label{lem:sep2}
Fix $ s\ge1 $ and $N>s+2r$. For general choices of $(a_\hh a_s,b_1\hh b_r)\in\ba^{s+r}(k)$ and $c\in k$, the polynomial
$f(x)=c\cd x^N+\prod_{i=1}^{s}(x-a_i)\prod_{j=1}^r(x-b_j)^2$ is separable. 
\end{lem}

\begin{proof}
For a fixed choice of $a_1\hh a_s,b_1\hh b_r$, $S(f)$ is a polynomial in $c$ of degree $N$ whose leading term is $N^Ne^{N-1}c^N$,
where $e=\prod a_i\cd \prod b_j^2$. (As the $a_i$ and $b_j$ are general, we have $e\neq 0$.) 
Hence $S(f)\neq 0$ for a general value of $c$.
\end{proof}

\section{Flynn's models}\label{sec:flynn}
In \cite{flynn1991sequences}, Flynn developed effective methods for producing rational torsion points of large orders on special families of hyperelliptic curves. In this section, we use Flynn's models in order to certify the realizability of certain multiplication profiles. 

\subsection{Torsion orders that grow linearly with respect to the genus}
In \textit{loc. cit.} Flynn recursively defined sequences of polynomials $\phi_r(x),\theta_r(x)$ given by
{\small
\[\theta_{r+1}(x)=(x+2)\theta_{r}(x)+2(x+1)\phi_r(x),\phi_{r+1}(x)=2\theta_{r}(x)+(x+2)\phi_r(x)\]
}
subject to $\theta_1(x)=\phi_1(x)=1$; and studied genus-$\ga$ hyperelliptic curves with presentations
{\small
\begin{equation}\label{eqn:flynn_presentation}
y^2+(\phi_r(x) y-\theta_r(x)x^{\ga})w(x)=x^{2\ga+1}+x^{2\ga}    
\end{equation}
}
where $w(x)$ is a polynomial of degree at most $\ga-2r+1$ and $w(0)\notin \{0,-4^{2-r}\}$. Flynn showed that for $p=(0,0)$ and $q$ equal to the point over infinity, $p-q$ is a torsion element in the Jacobian of \eqref{eqn:flynn_presentation} of order $2g+2r-1$ (\cite[Result 1(a)]{flynn1991sequences}).

\medskip
Our goal is to determine the multiplication profile of $p-q$. By symmetry, this is tantamount to computing $\ep_{\ga+i}$, for $i=2\hh r$ (though sometimes the profile will be determined by a smaller subset of the $\ep_j$).
When $i$ is small, the task further reduces to calculating the degree of $v(x)$, where $(x^{\ga+i},v(x))$ is the Mumford presentation for $(\ga+i)\cd(p-q)$. For example, when $i=2$, we have $\ep_{\ga+2}=\ga$ if $\deg(v)=\ga+1$ and $\ep_{\ga+2}=\ga-1$ if $\deg(v)\le\ga$. Indeed, the reduction algorithm terminates in one step in both cases and $\ep_{\ga+2}=\deg(\frac{v^2+vh-f}{x^{\ga+2}})$. Similarly, let $(x^{\ga+3},v(x))$ be the Mumford presentation of $(\ga+3)(p-q)$. Then $\ep_{\ga+3}=\ga$ if $\deg(v)=\ga+2$ and $\ep_{\ga+3}\le \ga-1$ otherwise; either way, the reduction algorithm terminates in two steps.   

\begin{nt}
Given a polynomial $f(x)$, let $f(x)_i$ denote the $x^i$-coefficient of its expansion.
\end{nt}
\noindent To compute the degree of $v(x)$ in these cases, we first solve the equation 
\[Z^2+\phi_r(x)w(x) Z-\theta_r(x)w(x)x^{\ga}=0, \qquad i=2,3\] 
over $k[\![x]\!]$, and then appropriately truncate the result. Since the discriminant of the left-hand side is $\phi_r^2w^2+4\theta_rwx^{\ga}$, 
it has a square root of the form $\phi_r(x)w(x)+\sum_{i=\ga}^{\infty}a_ix^i$. This, in turn, leads to the following equations:
{\small
\[\begin{cases}
2(\phi_rw)_0\cd a_{\ga}=4(\theta_rw)_0\\
2(\phi_rw)_1\cd a_{\ga}+2(\phi_rw)_0\cd a_{\ga+1}=4(\theta_rw)_1\\
2(\phi_rw)_2\cd a_{\ga}+2(\phi_rw)_1\cd a_{\ga+1}+2(\phi_rw)_0\cd a_{\ga+2}=4(\theta_rw)_2\\
\hdots
\end{cases}
\]
}
\noindent Using the inductive definitions of $\phi$ and $\theta$, respectively, we deduce that 
{\small
\[(\phi_r)_0=(\theta_r)_0=4^{r-1} \text{ and } (\theta_{r+1})_k-(\phi_{r+1})_k=(\theta_{r})_{k-1}+(\phi_{r})_{k-1}, 1\le k\le r.\]
}
Thus, $a_{\ga}=2$,  $a_{\ga+1}=\frac{2((\theta_{r}w)_1-(\phi_{r}w)_1)}{(\phi_rw)_0}=\frac{2((\theta_{r-1}w)_0+(\phi_{r-1}w)_0)}{(\phi_rw)_0}=1$, and
{\small
\[2(\phi_r)_1(w)_0+8(\phi_{r-1})_0(w)_1+2(\phi_rw)_0a_{\ga+2}=4[(\theta_rw)_2-(\phi_rw)_2].\]
}
In view of 
{\small
\[
(\theta_rw)_2-(\phi_rw)_2=[(\theta_r)_2-(\phi_r)_2]w_0+[(\theta_r)_1-(\phi_r)_]w_1=[(\theta_{r-1})_1+(\phi_{r-1})_1]w_0+[(\theta_{r-1})_0+(\phi_{r-1})_0]w_1
\]
}
it follows that
{\small
\[2(\phi_r)_1(w)_0+2(\phi_rw)_0a_{\ga+2}=4[(\theta_{r-1})_1+(\phi_{r-1})_1]w_0.\]
}
Finally, applying the relation $(\phi_r)_1=2(\theta_{r-1})_1+2(\phi_{r-1})_1+(\phi_{r-1})_0$, we deduce that $2(\phi_{r-1}w)_0+2(\phi_{r}w)_0a_{\ga+2}=0$, so $a_{\ga+2}=-\frac{1}{4}$. Since we are merely interested in realizing 
multiplication profiles, we may fix $w(x)=1$ (and thereby remove the restriction placed on $\ga$ in \cite{flynn1991sequences}). Letting $r=4,5$ yields:
\begin{cor}
For $r=4$ and $5$, let $p=(0,0)$ be the point defined as above. When $r=4$ and $\ga\ge 3$ (resp. $r=5$ and $\ga\ge 4$), $p-q$ has order $2\ga+7$ (resp. $2\ga+9$) and $m(p-q)$ is the multiplication profile determined by $(\ep_{\ga+2},\ep_{\ga+3})=(\ga,\ga)$ (resp. $(\ep_{\ga+2},\ep_{\ga+3},\ep_{\ga+4})=(\ga,\ga,\ga)$).
\end{cor}

\subsection{A variant of Flynn's construction}
Assume $\ga\ge 4$. Let $B$ denote the hyperelliptic affinely presented by $y^2+(x^3+4)y=x^{2\ga+1}+x^{2\ga-2}-(3x^3+4)x^{\ga-1}$. Let $p=(0,0)$ and $q$ be the point over infinity. We claim that the principal divisor associated to the rational function $(3x^3+4)x^{\ga-1}+(x^3+4)y$ on $B$ is $(2\ga+7)(p-q)$.

\medskip
\noindent To begin, note that when $(3x^3+4)x^{\ga-1}+(x^3+4)y=0$, we have 
\[(3x^3+4)^2x^{2\ga-2}-(x^3+4)^2y^2=(3x^3+4)^2x^{2\ga-2}-(x^3+4)^2(x^{2\ga+1}+x^{2\ga-2})=-x^{2\ga+7}=0.\]
As a result, the given function vanishes only at $p=(0,0)$, and its associated principal divisor is
$\dv((3x^3+4)x^{\ga-1}+(x^3+4)y)=(2\ga+7)(p-q)$. Now $\iota(p)=(0,-4)$; so $p$ is not a Weierstrass point. Furthermore $2\ga+7$ is not divisible by $2\ga+i$, $1\le i\le 6$ when $\ga\ge 2$; so $p-q$ is of order exactly $2\ga+7$.

\medskip
To determine $m(p-q)$, we solve the quadratic equation $Z^2+(x^3+4)Z-16x^{\ga-1}=0$ over $k[\![x]\!]$.\footnote{This is where we require that $\ga \geq 4$.} A square root of the discriminant is of the form $x^3+4+\sum_{i\ge\ga-1}a_ix^i$, which leads to the following equations:
{\small
\[\begin{cases}
2\cd 4\cd a_{\ga-1}=4\cd(-4)\\
2\cd 0\cd a_{\ga-1}+2\cd 4\cd a_{\ga}=0 \\
2\cd 0\cd a_{\ga-1}+2\cd 0\cd a_{\ga}+2\cd 4\cd a_{\ga+1}=0 \\
\hdots\hdots
\end{cases}
\]
}
Hence, $a_{\ga-1}=-2$ and $a_{\ga}=a_{\ga+1}=0$. It follows that there is some $v(x)$ that solves the congruence
{\small
\[Z^2+(x^3+4)Z-[x^{2\ga+1}+x^{2\ga-2}-(3x^3+4)x^{\ga-1}]\equiv 0\qquad(\text{mod }x^{\ga+2})\] 
}
and for which $\deg v(x)=\ga-1$. So $m(p-q)$ is the second profile in Ex.~\ref{ex:odd1}.  
\begin{rem}
Technically, we should check that the discriminant of $(x^3+4)-4[(3x^3+4)^2x^{2\ga-2}-(x^{2\ga+1}+x^{2\ga-2})]$ is non-zero. However, as the curve $B$ is non-singular at every point involved in the putative linear equivalence relations, it suffices to replace $B$ by its normalization, if necessary, to certify 
the existence of a torsion element with the desired multiplication profile.
\footnote{More explicitly, let $F(x,y)=y^2+(x^3+4)y+(3x^3+4)x^{\ga-1}-x^{2\ga+1}-x^{2\ga-2}$. We have $\frac{\p F}{\p y}(0,0)=4$, so the curve cut out by $F(x,y)$ is non-singular at $p$. Computing the principal divisor of $-x^{\ga-1}-y$ yields $(\ga+2)(p-q)\sim \sum_{i=1}^{\ga-1}(r_i-q)$, where $r_i=(\xi_i,-\xi_i^3-2)$ and the $\xi_i$ run through all roots of $x^{\ga-1}-2=0$. It is easy to see $\frac{\p F}{\p y}(r_i)=-\xi_i^3\neq 0$.} 
\end{rem}

\end{document}